\documentclass[11pt]{article}
\usepackage{amsfonts}
\usepackage{amssymb}
\usepackage{amsbsy}
\usepackage{amsthm}
\usepackage{graphicx}
\usepackage{color}

\newtheorem{thm}{Theorem}[section]
\newtheorem{cor}[thm]{Corollary}
\newtheorem{lem}[thm]{Lemma}
\newtheorem{prop}[thm]{Proposition}
\theoremstyle{definition}
\newtheorem{defn}[thm]{Definition}
\theoremstyle{remark}
\newtheorem{rem}[thm]{Remark}

\newcommand{\bt}{\begin{thm}}
\newcommand{\et}{\end{thm}}
\newcommand{\bc}{\begin{cor}}
\newcommand{\ec}{\end{cor}}
\newcommand{\bl}{\begin{lem}}
\newcommand{\el}{\end{lem}}
\newcommand{\bp}{\begin{prop}}
\newcommand{\ep}{\end{prop}}
\newcommand{\bd}{\begin{defn}}
\newcommand{\ed}{\end{defn}}
\newcommand{\br}{\begin{rem}}
\newcommand{\er}{\end{rem}}
\newcommand{\bpr}{\begin{proof}}
\newcommand{\epr}{\end{proof}}

\newcommand{\bi}{\begin{itemize}}
\newcommand{\ei}{\end{itemize}}
\newcommand{\be}{\begin{enumerate}}
\newcommand{\ee}{\end{enumerate}}
\newcommand{\ds}{\displaystyle}
\newcommand{\ba}{\begin{array}}
\newcommand{\ea}{\end{array}}
\newcommand{\beq}{\begin{equation}}
\newcommand{\eeq}{\end{equation}}
\newcommand{\beqa}{\begin{eqnarray}}
\newcommand{\eeqa}{\end{eqnarray}}

\newcommand{\N}{{\mathbb N}}
\newcommand{\Z}{{\mathbb Z}}
\newcommand{\Q}{{\mathbb Q}}
\newcommand{\R}{{\mathbb R}}
\newcommand{\C}{{\mathbb C}}
\newcommand{\T}{{\mathbb T}}
\newcommand{\D}{{\mathbb D}}
\newcommand{\PP}{{\mathbb P}}

\newcommand{\J}{{\mathbb J}}

\newcommand{\cC}{{\mathcal  C}}

\newcommand{\cX}{{\mathcal  X}}

\newcommand{\spn}{\mathrm{span}}

\newcommand{\re}{\mathrm{Re}}
\newcommand{\im}{\mathrm{Im}}
\newcommand{\adj}{\mathrm{Adj}}
\newcommand{\reg}{\mathrm{reg}}
\newcommand{\rank}{\mathrm{rank}}

\begin{document}

\title{\bf Polynomial perturbations of hermitian linear functionals and difference equations}
\author{M.J. Cantero, L. Moral, L. Vel\'azquez
\footnote{The work of the authors was partially supported by the Spanish grants from the Ministry of Education and Science, project code
MTM2005-08648-C02-01, and the Ministry of Science and Innovation, project code MTM2008-06689-C02-01, and by Project E-64 of Diputaci\'{o}n General de
Arag\'{o}n (Spain).}}
\date{\small Departamento de Matem\'atica Aplicada. Universidad de Zaragoza. Spain}
\maketitle

\begin{abstract}

This paper is devoted to the study of general (Laurent) polynomial modifications of moment functionals on the unit circle, i.e., associated with
hermitian Toeplitz matrices. We present a new approach which allows us to study polynomial modifications of arbitrary degree.

The main objective is the characterization of the quasi-definiteness of the functionals involved in the problem in terms of a difference equation
relating the corresponding Schur parameters. The results are presented in the general framework of (non necessarily quasi-definite) hermitian
functionals, so that the maximum number of orthogonal polynomials is characterized by the number of consistent steps of an algorithm based on the
referred recurrence for the Schur parameters.

Some concrete applications to the study of orthogonal polynomials on the unit circle show the effectiveness of this new approach: an exhaustive and
instructive analysis of the functionals coming from a general inverse polynomial perturbation of degree one for the Lebesgue measure; the
classification of those pairs of orthogonal polynomials connected by a kind of linear relation with constant polynomial coefficients; and the
determination of those orthogonal polynomials whose associated ones are related to a degree one polynomial modification of the original orthogonality
functional.

\end{abstract}


\medskip

\noindent{\it (2000) AMS Mathematics Subject Classification}: 42C05.

\section{Introduction} \label{I}

The intense activity during the last decades around the theory of orthogonal polynomials on the unit circle has stimulated the study of perturbations
of hermitian functionals. The possibility of considering modifications that do not preserve the hermitian character of the functional leads to left
and right orthogonal polynomials (see \cite{B}), thus most of the efforts have been concentrated in the analysis of hermitian perturbations as a
source of new families of standard orthogonal polynomials (see the recent monograph on orthogonal polynomials on the unit circle \cite{Si105,Si205})
and the references therein).

This paper proposes a new method to study the hermitian modifications obtained when multiplying a hermitian functional by a Laurent polynomial of any
degree, in short, the hermitian polynomial modifications. This kind of perturbation has been considered previously (see for instance
\cite{ACMV,Ca,CaGaMa,DaHeMa07,GaHeMa08,GaHeMa09,Ga09,GaMa1,GaMa2,GaMa3,God,Go91,GodMarc,LiMa99,MaHe05,MaHe08,Su93}), but the usual approaches have
the drawback of being formulated in terms of orthogonal polynomials, kernels and determinants, what makes difficult the practical application,
specially for perturbations of high degree.

On the contrary, our method is based on a recurrence for the Schur parameters of the two functionals involved in the perturbation. This provides an
algorithm to generate the Schur parameters of one of the functionals, starting from the Schur parameters of the other functional. Furthermore, this
recurrence yields a characterization of the maximum number of orthogonal polynomials for one of the functionals, given the number of orthogonal
polynomials that the other functional has. That is, this approach permits us to study the relation between the quasi-definiteness of a functional and
a polynomial modification of any degree with less computational effort than the methods already existing.

We distinguish between three different but related problems, depending on the data at hand:
\begin{itemize}
\item Basic problem: characterize when two functionals are related by a polynomial perturbation in terms of their Schur parameters.
\item Direct problem: characterize the quasi-definiteness of a polynomial modification from the Schur parameters of the original functional.
\item Inverse problem: characterize the quasi-definiteness of a functional from the Schur parameters of one of its polynomial modifications.
\end{itemize}

Despite the symmetry between the direct and inverse problems, they have a quite different nature which makes much more interesting the last one. The
root of this difference is the fact that, given a functional and a Laurent polynomial, the corresponding polynomial modification is uniquely defined
while there are infinitely many functionals whose modification is the given one. This leads to a rich structure in the set of solutions of the
inverse problem which, as we will see, is related to another kind of interesting modifications: the addition of Dirac deltas and its derivatives.
Hence, any information about inverse polynomial modifications can be translated as a result on perturbations by Dirac deltas.

Furthermore, as an example will show, some special solutions of an inverse problem can act as ``attractors" for the asymptotics of the parameters of
other solutions. Thus, the analysis of those special solutions provides information about the asymptotics of perturbations by Dirac deltas.

The rich structure of the inverse problem has a double interest due to the fact that our approach, based on a recurrence for the Schur parameters,
also yields interesting connections between the study of polynomial modifications and difference equations. Therefore, the asymptotics of the
solutions of the inverse problem is closely related to the asymptotics of difference equations.

The content of the paper is structured in the following way: the rest of the introduction summarizes the basic definitions and notations; Section
\ref{HPM} includes the main results about hermitian polynomial modifications, i.e., it is devoted to what we call basic problem; direct and inverse
problems are discussed in Section \ref{DIP}, including an exhaustive analysis of an explicit example of inverse problem; and Section \ref{A} shows
other applications of the techniques developed in the paper, i.e., a complete classification of the pairs of orthogonal polynomials related by
certain type of linear relations with constant polynomial coefficients, and the determination of the orthogonal polynomials whose associated ones
come from a polynomial modification of degree one of the original orthogonality functional.

Now we proceed with the conventions for the notation.

In what follows $\T:=\{z\in\C : |z|=1\}$ and $\D:=\{z\in\C : |z|<1\}$ are called respectively the unit circle and the open unit disk on the complex
plane. $\PP:=\C[z]$ is the complex vector space of polynomials with complex coefficients, and $\PP_n$ the vector subspace of polynomials whose degree
is not greater than $n$, while $\PP_{-1} := \{0\}$ is the trivial subspace. $\Lambda:= \C[z, z^{-1}] $ is the complex vector space of Laurent
polynomials and, for $m\leq n$, we define the vector subspace $\Lambda_{m,n}:= \spn\{ z^m, z^{m+1},\dots, z^n \}$. Given any $f\in\Lambda$ we define
$f_*(z)=\overline{f}(z^{-1})$ and, if $p\in\PP_n\setminus\PP_{n-1}$, $p^*$ denotes its reversed polynomial $p^*(z) = z^n p_*(z)$. Sometimes we use
the notation $p^*(z)=z^np_*(z)$ for polynomials $p\in\PP_n$ whose degree can be smaller than $n$. Then we refer to the $*_n$ operator when it is
advisable to avoid misunderstandings.

Any hermitian linear functional $v$ on $\Lambda$ ($v[z^{-n}] = \overline{v[z^n]}$, $n=0,1,\dots$) defines a sesquilinear functional
$(\cdot\;,\cdot)_v\colon\Lambda\times\Lambda\longrightarrow\C$ by
$$
(f,g)_v:= v[f_*g], \qquad f,g\in\Lambda.
$$
The sequence $(p_n)_{n\geq 0}$ is a sequence of orthogonal polynomials with respect to the hermitian linear functional $v$ if
\begin{itemize}
\item[(i)] $p_n\in\PP_n\setminus\PP_{n-1}$,
\item[(ii)] $(p_n, p_m)_v = l_n \delta_{n,m}, \quad l_n\neq 0$,
\end{itemize}
and when such a sequence exists $v$ is called a quasi-definite functional. If $v[1]\neq0$ we can assure only the existence of a finite segment of
orthogonal polynomials, i.e., a finite set $(p_k)_{k=0}^n$ of polynomials satisfying (i) and (ii). When $v$ has a finite segment of orthogonal
polynomials $(p_k)_{k=0}^n$ of length $n+1$ we say that $v$ is quasi-definite on $\PP_n$.

In the positive definite case ($l_n>0$, $n=0,1,\dots$) there exists a positive measure $\mu$ supported on $\T$ providing an integral representation
for the functional $v$,
$$
v[f] = \int_{\T} f(z) d\mu(z), \qquad f\in\Lambda.
$$
Due to this reason a sequence $(p_n)_{n\geq0}$ satisfying (i) and (ii) is called a sequence of orthogonal polynomials on the unit circle, even in the
general quasi-definite case. If $l_n=\pm 1$ for all $n$, $(p_n)$ is called a sequence of orthonormal polynomials on the unit circle. We denote by
$(\hat{p}_n)_{n\geq 0}$ the orthonormal polynomials with positive leading coefficients.

In that follows $(\psi_n)_{n\geq0}$ denotes the sequence of monic orthogonal polynomials (MOP) with respect to a hermitian functional $v$. Two
hermitian linear functionals $v_1, v_2$ have a common finite segment $(\psi_j)_{j=0}^n$ of MOP iff there exists $\lambda\in\R^*$ such that $v_1[f] =
\lambda v_2[f]$ for any $f\in\Lambda_{-n,n}$, although requiring this condition to hold only for any $f\in\PP_n$ is enough due to the hermiticity. In
this case we say that $v_1$ and $v_2$ are equivalent in $\PP_n$ or, in a more symbolic way, $v_1\equiv v_2$ in $\PP_n$. If this holds for any $n$, we
simply say that $v_1$ and $v_2$ are equivalent and we write $v_1 \equiv v_2$.

A sequence $(\psi_n)$ is a sequence of MOP on the unit circle iff it satisfies the recurrence relation (see \cite{Sz,Ge2,Si105})
\begin{equation} \label{RR}
\psi_n(z) = z \psi_{n-1}(z) + \psi_n(0) \psi_{n-1}^*(z), \qquad n=1,2\dots,
\end{equation}
with $\psi_0(z)=1$ and $|\psi_n(0)|\neq1$ for $n\geq1$. Applying the $*_n$ operator to the above recurrence we get the equivalent one
\begin{equation} \label{RRR}
\psi_n^*(z) = \overline{\psi_n(0)} z \psi_{n-1}(z) + \psi_{n-1}^*(z), \qquad n=1,2\dots.
\end{equation}
The values $\psi_n(0)$ are called the Schur parameters or reflection coefficients of the hermitian linear functional $v$.

A straightforward computation yields
$$
1-|\psi_n(0)|^2 = {\varepsilon_n \over \varepsilon_{n-1}}, \qquad n=1,2,\dots,
$$
where $\varepsilon_n:=(\psi_n,\psi_n)_v=v[\psi_nz^{-n}]$ relates $\hat{p}_n$ and $\psi_n$ by $\hat{p}_n = |\varepsilon_n|^{-{1\over 2}} \psi_n.$ When
$v$ is positive definite $\varepsilon_n=\|\psi_n\|^2_{L^2(\mu)} > 0$ for $n\geq0$, which means that $|\psi_n(0)|<1$ for $n\geq1$.

\section{Hermitian polynomial modifications} \label{HPM}

We are interested in those (Laurent) polynomial modifications of hermitian functionals which preserve their hermitian character, in short, the
hermitian polynomial modifications of hermitian functionals. If $v$ is a linear functional on $\Lambda$ and $L\in\Lambda$ the modified functional
$vL$ is defined by
$$
vL[f]:=v[Lf], \qquad f\in\Lambda.
$$
The modified functional $vL$ is hermitian for every hermitian $v$ iff $L_*=L$, which is equivalent to state that $L=P+P_*$ with $P\in\PP$ (see
\cite{AM}). Such a polynomial $P$ can be uniquely determined by $L$ simply requiring $P(0)\in\R$, a convention that we will assume in what follows.
We will refer to $\deg P$ as the degree of the polynomial modification, which we will consider greater than or equal to one, and $L$ will be called a
hermitian Laurent polynomial of degree $r$.

Another way to characterize a hermitian polynomial modification is through the polynomial $A=z^{\deg P}L$ of degree $2\deg P$. The condition $L_*=L$
means that $A$ is self reciprocal, i.e., $A^*=A$. Thus the hermitian polynomial modifications are related to the self-reciprocal
polynomials of even degree.

The set of roots of a self-reciprocal polynomial, counting the multiplicity, is invariant under the transformation $\zeta \to 1/\overline{\zeta}$.
That is, their roots lie on the unit circle or appear in symmetric pairs $\zeta,1/\overline{\zeta}$. Indeed, this property characterizes the
self-reciprocal polynomials up to numerical factors. This implies that any self-reciprocal polynomial of even degree factorizes into a product of
self-reciprocal polynomials of degree 2. As a consequence, an arbitrary hermitian polynomial modification is a composition of elementary ones of
degree 1, i.e., if $L=P+P_*$ with $\deg P = r$, then $L = L_1 L_2 \cdots L_r$ with $L_k=P_k+P_{k*}$ and $\deg P_k=1$.

Sometimes we will deal with polynomials $A\in\PP_n$ whose degree is not necessarily $n$ but such that $A^{*_n}=A$. In this case we will say that $A$
is self-reciprocal in $\PP_n$ to avoid misunderstandings. Such a polynomial has the general form $A(z)=z^sB(z)$ where $B$ is strictly
self-reciprocal. Thus, a self-reciprocal polynomial in $\PP_n$ is actually self-reciprocal iff it has no zeros at the origin.

Given a hermitian functional $v$ and a Laurent polynomial $L=P+P_*$, our purpose is to obtain relations between the MOP and Schur parameters
associated with the functionals $v$ and $vL$. Multiplying $L$ by a non null real factor gives rise to a hermitian functional which is equivalent to
$vL$ and, hence, with the same MOP and Schur parameters as $vL$. Therefore, concerning our aim, the Laurent polynomial $L$, as well as the
polynomials $P$ and $A$, are defined up to non null real factors.

The following general result will be useful to achieve our objective. In what follows we denote by $S^{\bot_n}$ the orthogonal complement in $\PP_n$
of a subspace $S \subset \PP_n$.

\begin{lem}[see \cite{CTasis}] \label{BASISOC}

Let $v$ be a hermitian functional such that the corresponding $n$-th MOP $\psi_n$ exists. Then, $\mathfrak{B}= \{z^k\psi_n\}_{k=0}^r \cup
\{z^k\psi_n^*\}_{k=0}^{r-1}$ is a basis of  $\bigl(z^r\mathbb{P}_{n-r-1} \bigr)^{\bot_{n+r}}$ for $n \geq r \geq 1$, and a generator system of
$\PP_{n+r}$ for $r > n \geq 0$.

\end{lem}

\noindent {\it Sketch of the proof.} If $n \geq r \geq 1$, the orthogonality of $\psi_n$ assures that $\mathfrak{B} \subset
\bigl(z^r\mathbb{P}_{n-r-1} \bigr)^{\bot_{n+r}}.$ Besides, $\Omega \in \spn \mathfrak{B}$ iff $\Omega = C \psi_n + D \psi_n^*$, $C\in\PP_r$,
$D\in\PP_{r-1}$. Furthermore, this decomposition is unique because $\gcd (\psi_n, \psi_n^*) = 1,$ which proves the linear independence of
$\mathfrak{B}.$ Then, the first result follows from the fact that $\sharp\mathfrak{B} = 2r+1 = \dim \bigl(z^r\mathbb{P}_{n-r-1} \bigr)^{\bot_{n+r}}$.

Suppose now that $r > n \geq0$. From the previous result we know that $\{z^k\psi_n\}_{k=0}^n \cup \{z^k\psi_n^*\}_{k=0}^{n-1}$ is a basis of
$\mathbb{P}_{2n}$. Hence, $\{z^k\psi_n\}_{k=0}^r \cup \{z^k\psi_n^*\}_{k=0}^{n-1}$ is a linear independent subset of $\PP_{n+r}$ with $n+r+1$
elements, thus it is a basis of $\PP_{n+r}$, which proves the second result. \hfill{$\Box$}

\medskip

Our interest in the previous lemma is the following direct consequence.

\begin{cor} \label{BASISC}

Let $v$ be a hermitian functional such that the corresponding $n$-th MOP $\psi_n$ exists. Then, every polynomial
$\Omega\in\bigl(z^r\mathbb{P}_{n-r-1} \bigr)^{\bot_{n+r}}$ has a unique decomposition $\Omega = C \psi_n + D \psi_n^*$, $C\in\PP_r$, $D\in\PP_{r-1}$,
for $n \geq r \geq 1$, and every polynomial $\Omega\in\PP_{n+r}$ has infinitely many such decompositions for $r > n \geq 0$.

\end{cor}

\begin{rem} \label{BASISNR}

It is worth it to remark the case $n=r$ in the above corollary, which says that every polynomial $\Omega\in \mathbb{P}_{2r}$ admits a unique
decomposition $\Omega = C \psi_r + D \psi_r^*$, $C\in\PP_r$, $D\in\PP_{r-1}$.

\end{rem}

The next theorem is the starting point for our approach to the study of hermitian polynomial modifications of hermitian functionals.

\begin{thm} \label{GENERALPROBLEMTH}

Let $u$, $v$ be hermitian functionals with finite segments of MOP $(\varphi_j)_{j=0}^n$, $(\psi_j)_{j=0}^{n+r}$ respectively, and let
$L=P+P_*=z^{-r}A$ with $P$ a polynomial of degree $r$. Then, the following statements are equivalent:
\begin{itemize}
\item[(i)] $u\equiv vL$ in $\PP_n$.
\item[(ii)] There exist $C_j\in\PP_r$, $D_j\in\PP_{r-1}$ with $C_j(0)\neq 0$ such that
\begin{equation} \label{GPOLO}
A \varphi_j = C_j \psi_{j+r} + D_j \psi_{j+r}^*, \qquad j=0,\dots,n.
\end{equation}
\item[(iii)] There exist $C_j\in\PP_r$, $D_j\in\PP_{r-1}$ with $C_j(0)\neq 0$ such that
\begin{equation} \label{GPOLO1}
A \varphi_j^* =  z D_j^* \psi_{j+r} + C_j^* \psi_{j+r}^*, \qquad D_j^*=D_j^{*_{r-1}}, \qquad j=0,\dots,n.
\end{equation}
\end{itemize}
The polynomials $C_j\in\PP_r$, $D_j\in\PP_{r-1}$ satisfying $(\ref{GPOLO})$ or $(\ref{GPOLO1})$ are unique, $\deg C_j = r$, $C_j(0)\in\R$ and
$C_j^*(0)=A(0)$.

\end{thm}

\begin{proof}
The equivalence between (ii) and (iii) follows from the use of the $*_{2r+j}$ operator and the fact that $A$ is a self-reciprocal polynomial of
degree $2r$. Also, assuming (ii) we get $\deg C_j = r$ because $\deg(D_j\psi_{j+r}^*)<\deg(A\varphi_j)=2r+j$, and the equality
$(\varphi_j,\varphi_j)_u = u\bigl[\varphi_j z^{-j}\bigl] = C_j(0) \varepsilon_{j+r}$ implies $C_j(0)\in\R$. On the other hand, evaluating
(\ref{GPOLO1}) at $z=0$ we find that $C_j^*(0)=A(0)$. It only remains to prove the equivalence between (i) and (ii) and the uniqueness of
decomposition (\ref{GPOLO}).

Suppose (i), i.e., $u[f] = \lambda vL[f]$, $\lambda\in\R^*$, for any $f\in\Lambda_{-n,n}$. The orthogonality of $(\varphi_j)_{j=0}^n$ with respect to
$u$ gives
$$
0 = u\bigl[ \varphi_j z^{-k}\bigr] = \lambda v\bigl[A \varphi_j z^{-(k+r)}\bigr], \qquad r\leq k+r \leq j+r-1,
$$
which means that $A \varphi_j \in \bigl(z^r\mathbb{P}_{j-1} \bigr)^{\bot_{2r+j}}$ with respect to $v$. Using Corollary \ref{BASISC} we get
(\ref{GPOLO}) and the uniqueness of the polynomials $C_j$, $D_j$.

On the other hand, if $(\varphi_j)_{j=0}^n$, $(\psi_j)_{j=0}^{n+r}$ satisfy (\ref{GPOLO}), the orthogonality of $(\psi_j)_{j=0}^{n+r}$ with respect
to $v$ yields
$$
vL[\varphi_j  z^{-k}\bigl] = v\bigl[A\varphi_j  z^{-(k+r)}\bigl]= v\bigl[\bigl(C_j \psi_{j+r} + D_j \psi_{j+r}^*\bigr) z^{-(k+r)}\bigr] = 0
$$
for $0\leq k\leq j-1$ and
$$
vL\bigl[\varphi_j z^{-j}\bigl] = v\bigl[A\varphi_j z^{-(j+r)}\bigl] = v\bigl[\bigl(C_j \psi_{j+r} + D_j \psi_{j+r}^*\bigr)z^{-(j+r)}\bigr] =
C_j(0)\varepsilon_{j+r}.
$$
So, $C_j(0)\neq0$ for $j=0,\dots,n$ iff $(\varphi_j)_{j=0}^n$ is a finite segment of MOP with respect to $vL$, which means that $u \equiv vL$ in
$\PP_n$.
\end{proof}

Equality (\ref{GPOLO1}) is true taking $D_j^* = D_j^{*_{r-1}}$, no matter whether $D_j$ has degree $r-1$ or not. In what follows we will assume this
convention for the polynomials $D_j$.

\begin{rem} \label{GENERALPROBLEMRMRK}

The functional $u$ has a finite segment of MOP of length (at least) one iff $u[1] \neq 0$. Therefore, Theorem \ref{GENERALPROBLEMTH} assures that the
condition $v[L] \neq 0$ is equivalent to the existence of a (unique) decomposition
\begin{equation} \label{INITIALCONDITION}
A = C_0 \psi_r + D_0 \psi_r^*, \qquad C_0\in\PP_r, \qquad D_0\in\PP_{r-1},
\end{equation}
with $C_0(0) \neq 0$. However, Remark \ref{BASISNR} says even more: no matter the value of $v[L]$, there is always a unique decomposition like
(\ref{INITIALCONDITION}). The equality $v[L] = C_0(0)\varepsilon_r$ implies that  $v[L]\neq  0$ is only responsible of $C_0(0)\neq0$.

\end{rem}

The above theorem has the following consequence for quasi-definite functionals.

\begin{cor} \label{DIRECTPROBLEMCOR}

Let $u$, $v$ be quasi-definite functionals with sequences of MOP $(\varphi_n)$, $(\psi_n)$ respectively, and let $L=P+P_*=z^{-r}A$ with $P$ a
polynomial of degree $r$. Then, $u \equiv vL$ iff there exist polynomials $C_n\in\PP_r$, $D_n\in\PP_{r-1}$ with $C_n(0)\not= 0$ such that
\begin{equation} \label{QDPOL}
A \varphi_n = C_n \psi_{n+r} + D_n \psi_{n+r}^*, \qquad n\geq 0,
\end{equation}
or equivalently
$$
A \varphi_n^* =  z D_n^* \psi_{n+r} + C_n^* \psi_{n+r}^*, \qquad n\geq 0.
$$

\end{cor}

For convenience, in what follows we will use a matrix notation and we will adopt some definitions and conventions that will be used in the rest of
the paper. If $L$ is a hermitian Laurent polynomial of degree $r$, $P$ and $A$ are the polynomials given by $L=P+P_*=z^{-r}A$, $P(0)\in\R$. We denote
by $\phi_j$ and $\psi_j$ the $j$-th MOP with respect to the hermitian functionals $u$ and $v$ respectively. Also,
$$
\begin{array}{c}
a_j = \varphi_j(0), \qquad b_j = \psi_j(0),
\qquad
e_j = (\varphi_j,\varphi_j)_u, \qquad \varepsilon_j = (\psi_j,\psi_j)_v,
\medskip \\
\Phi_j = \pmatrix{\varphi_j \cr \varphi_j^*},
\qquad
\mathcal{S}_j = \pmatrix{z & a_j \cr z\overline{a}_j & 1},
\qquad
\mathcal{A}_j = \pmatrix{1 & a_j \cr \overline{a}_j & 1},
\medskip \\
\Psi_j = \pmatrix{\psi_j \cr \psi_j^*},
\qquad
\mathcal{T}_j = \pmatrix{z & b_j \cr z\overline{b}_j & 1},
\qquad
\mathcal{B}_j = \pmatrix{1 & b_j \cr \overline{b}_j & 1},
\medskip \\
\mathcal{C}_j = \pmatrix{C_j & D_j \cr zD_j^* & C_j^*},
\qquad
\tilde{\mathcal{C}}_j = \pmatrix{C_j & zD_j \cr D_j^* & C_j^*}.
\end{array}
$$
The matrices $\mathcal{S}_j$ and $\mathcal{T}_j$, known as {\it transfer matrices}, permit us to write recurrence relations (\ref{RR}) and
(\ref{RRR}) for $(\varphi_n)$ and $(\psi_n)$ in the compact form
\begin{equation} \label{TM}
\Phi_j = \mathcal{S}_j \Phi_{j-1},
\qquad
\Psi_j = \mathcal{T}_j \Psi_{j-1},
\end{equation}
while the matrices $\mathcal{C}_j$ make possible to combine (\ref{GPOLO}) and (\ref{GPOLO1}) into
$$
A\Phi_j = \mathcal{C}_j\Psi_{j+r}.
$$
The structure of the matrices $\mathcal{C}_j$ is worth to be remarked.

\begin{defn} \label{J-M}

A polynomial matrix $\mathcal{C} = \scriptsize \pmatrix{C_1 & D_1 \cr D_2 & C_2}$, $C_i,D_i\in\PP_ r$, satisfying ${\mathcal{C}}^{*_r} = J
\mathcal{C} J$ with $J=\scriptsize\pmatrix{0&1\cr1&0}$ will be called a $J$-{\it self-reciprocal matrix} in $\PP_r$. This is equivalent to state that
$C_2=C_1^{*_r}$ and $D_2=D_1^{*_r}$.

We denote by $\J_r$ the set of $J$-selfreciprocal matrices in $\PP_r$ such that $C_2(0)\neq0$ and $D_2(0)=0$. These conditions mean that $\deg C_1 =
r$ and $\deg D_1 \leq r-1$, thus the general form of a polynomial matrix $\mathcal{C}\in\J_r$ is
\begin{equation} \label{Cm}
\mathcal{C}=\pmatrix{C & D \cr zD^* & C^*}, \quad \deg C = r, \quad \deg D \leq r-1,
\end{equation}
where here and below we assume that $D^* = D^{*_{r-1}}$.

Given a polynomial matrix $\mathcal{C}\in\J_r$ like (\ref{Cm}) we will denote
$$
\tilde\mathcal{C} = \pmatrix{C & zD \cr D^* & C^*},
$$
which is $J$-self-reciprocal too, but in general does not necessarily belong to $\J_r$ because $zD$ can have degree $r$.

The determinant of a $J$-self-reciprocal matrix $\mathcal{C}$ in $\PP_r$ is a self-reciprocal polynomial in $\PP_{2r}$. When $\det\mathcal{C}$ has
degree $2r$ we will say that $\mathcal{C}$ is a {\it regular $J$-self-reciprocal matrix}. This is equivalent to $\det\mathcal{C}(0) \neq 0$, which in
case of $\mathcal{C}\in\J_r$ means simply $C(0)\neq 0$. We will denote by $\J_r^\reg$ the subset of regular $J$-self-reciprocal matrices of $\J_r$.

\end{defn}

The next result about $J$-self-reciprocal matrices will be useful later on.

\begin{lem} \label{CT=SC}

Let $\mathcal{S} = \pmatrix{z & a \cr z\overline{a} & 1}$, $\mathcal{T} = \pmatrix{z & b \cr z\overline{b} & 1}$ with $a,b\in\C$.
\begin{itemize}
\item[(i)] If $|a|\neq1$, $\mathcal{C}\in\J_r$, the equation $\mathcal{C}\mathcal{T}=\mathcal{S}\hat\mathcal{C}$ defines a matrix
$\hat\mathcal{C}\in\J_r$ iff
$$
a\,C^*(0)=b\,C(0)+D(0).
$$
In this case $\hat\mathcal{C}\in\J_r^\reg\;\Leftrightarrow\;|b|\neq1,\;\mathcal{C}\in\J_r^\reg$.
\item[(ii)] If $|b|\neq1$, $\mathcal{C}\in\J_r$, the equation $\hat\mathcal{C}\mathcal{T}=\mathcal{S}\mathcal{C}$ defines a matrix
$\hat\mathcal{C}\in\J_r$ iff
$$
a\,\overline{C(0)}=b\,\overline{C^*(0)}-\overline{D^*(0)}.
$$
In this case $\hat\mathcal{C}\in\J_r^\reg\;\Leftrightarrow\;|a|\neq1,\;\mathcal{C}\in\J_r^\reg$.
\end{itemize}

\end{lem}

\begin{proof}
If $|a|\neq1$ the equation $\mathcal{C}\mathcal{T}=\mathcal{S}\hat\mathcal{C}$ can be written as
$$
\hat\mathcal{C} =  \pmatrix{z^{-1}&0 \cr 0&1} \mathcal{X} \pmatrix{z&0 \cr 0&1}, \quad \mathcal{X} = \frac{1}{1-|a|^2} \pmatrix{1&-a \cr
-\overline{a}&1} \mathcal{C} \pmatrix{1&b \cr \overline{b}&1}.
$$
Let $\mathcal{C}\in\J_r$. Then $\mathcal{X}$ is a $J$-self-reciprocal matrix in $\PP_r$, i.e., $\mathcal{X} = \scriptsize \pmatrix{X&Y \cr
Y^{*_r}&X^{*_r}}$ with $X,Y\in\PP_r$. Therefore, $\hat\mathcal{C}$ is a polynomial matrix iff $Y(0)=0$, which yields the relation between $a$ and $b$
given in (i). In such a case $Y=z\hat Y$, $\hat Y\in\PP_{r-1}$, and $X^{*_r}(0)=C^*(0)\neq0$, thus $\hat\mathcal{C} = {\scriptsize\pmatrix{X&\hat Y
\cr z\hat Y^{*_{r-1}} & X^{*_r}}} \in \J_r$. Also, $X(0)=C(0)(1-|b|^2)/(1-|a|^2)$, hence $\hat\mathcal{C}\in\J_r^\reg \Leftrightarrow
|b|\neq1,\,\mathcal{C}\in\J_r^\reg$.

On the other hand, if $|b|\neq1$ the equation $\hat\mathcal{C}\mathcal{T}=\mathcal{S}\mathcal{C}$ reads as
$$
\hat\mathcal{C} = \frac{1}{1-|b|^2} \pmatrix{1&a \cr \overline{a}&1} \tilde\mathcal{C} \pmatrix{1&-b \cr -\overline{b}&1}.
$$
Suppose that $\mathcal{C}\in\J_r$. Then $\hat\mathcal{C}$ is a $J$-self-reciprocal matrix in $\PP_r$, hence $\hat\mathcal{C} = \scriptsize
\pmatrix{X&Y \cr Y^{*_r}&X^{*_r}}$ with $X,Y\in\PP_r$. The relation between $a$ and $b$ given in (ii) is equivalent to $Y^{*_r}(0)=0$, and also gives
$X^{*_r}(0)=C^*(0)\neq0$, $X(0)=C(0)(1-|a|^2)/(1-|b|^2)\neq0$, so $\hat\mathcal{C}\in\J_r$ and $\hat\mathcal{C}\in\J_r^\reg \Leftrightarrow
|a|\neq1,\,\mathcal{C}\in\J_r^\reg$.
\end{proof}

\medskip

The goal of the rest of the section is to present a more economical and effective approach than the ones already existing in the literature (see for
instance \cite{God,Go91,GodMarc}) to study the relation $u \equiv vL$ for any degree of $L$. This new point of view avoids the calculation of
determinants and MOP related to $u$ and $v$, requiring only the knowledge of the corresponding Schur parameters and the Laurent polynomial $L$. More
precisely, we will characterize the relation $u \equiv vL$ through a matrix difference equation for the Schur parameters involving
$J$-self-reciprocal matrices.

 The first step to formulate this new approach is to translate the relations between the MOP
 $(\varphi_n)$ and $(\psi_n)$ into relations between the corresponding Schur parameters. The
 following result will be useful for this purpose.

\begin{lem} \label{MATRICES}

Let $P$, $Q$ be relatively prime polynomials with $\deg Q \leq \deg P$. If the polynomial matrices
$$
M = \pmatrix{M_1 & M_2 \cr M_3 & M_4}, \qquad N = \pmatrix{N_1 & N_2 \cr N_3 & N_4},
$$
satisfy $\deg(M_2-N_2), \deg (M_4-N_4) < \deg P$, then
$$
M \pmatrix{P \cr Q} = N \pmatrix{ P \cr Q} \quad\Leftrightarrow\quad M=N.
$$

\end{lem}

\begin{proof}
$M_1P+M_2Q = N_1P+N_2Q$, thus $(M_1-N_1) P = (N_2-M_2)Q$. Since $\gcd (P,Q)=1$, necessarily $P$ divides $M_2-N_2$, which implies $M_2-N_2=0$ because
$\deg (M_2-N_2) < \deg P$. Therefore $M_1-N_1 = 0$ too. Analogously $M_3-N_3 = M_4-N_4 = 0$.
\end{proof}

The next result is the matrix form of Theorem \ref{GENERALPROBLEMTH}, together with a stronger result and some properties of the polynomial matrices
$\mathcal{C}_j$, including the first relations between the Schur parameters  $(a_n)$ and $(b_n)$.

\begin{thm} \label{DIRECTPROBLEMTHMF}

Let $u$, $v$ be quasi-definite in $\PP_n$, $\PP_{n+r}$ respectively and let $L$ be a hermitian Laurent polynomial of degree $r$. Then, the following
statements are equivalent:
\begin{itemize}
\item[(i)] $u \equiv vL$ in $\PP_n$.
\item[(ii)] There exist $\mathcal{C}_0,\dots,\mathcal{C}_n\in\J_r^\reg$ such that
\begin{equation} \label{GPRMF}
A \Phi_j = \mathcal{C}_j  \Psi_{j+r}, \qquad j=0,\dots,n.
\end{equation}
\item[(iii)] There exists $\mathcal{C}_n\in\J_r^\reg$ such that
\begin{equation} \label{GPRMF2}
A \Phi_n = \mathcal{C}_n  \Psi_{n+r}.
\end{equation}
\end{itemize}
The matrices $\mathcal{C}_j$ are the only solutions of $(\ref{GPRMF})$ in $\J_r$, so $\mathcal{C}_0$ is determined by
\begin{equation} \label{MATINITIALCONDITION}
\mathcal{C}_0 \Psi_r = A \pmatrix{1 \cr 1}, \qquad \mathcal{C}_0\in\J_r.
\end{equation}
Besides, we have the relations
\begin{equation} \label{MAT}
\mathcal{C}_j \mathcal{T}_{j+r} = \mathcal{S}_j \mathcal{C}_{j-1}, \qquad j=1,\dots,n,
\end{equation}
\begin{equation} \label{MATILDA}
\mathcal{C}_j \mathcal{B}_{j+r} = \mathcal{A}_j \tilde{\mathcal{C}}_{j-1}, \qquad j=1,\dots,n,
\end{equation}
\begin{equation} \label{DETAC}
\det\mathcal{C}_j = C_j(0) A, \qquad j=0,\dots,n.
\end{equation}

\end{thm}

\begin{proof}
Bearing in mind Theorem \ref{GENERALPROBLEMTH}, it is enough to prove (iii) $\Rightarrow$ (ii) $\Rightarrow$ (\ref{MAT}), (\ref{MATILDA}),
(\ref{DETAC}). Suppose that only (iii) holds. Evaluating (\ref{GPRMF2}) at $z=0$ we find $a_n A(0) = b_{n+r} C_n(0) + D_n(0)$ and $C_n^*(0)=A(0)$.
Hence, Lemma \ref{CT=SC} (i) assures the existence of $\mathcal{C}_{n-1}\in\J_r^\reg$ satisfying $\mathcal{C}_n \mathcal{T}_{n+r} = \mathcal{S}_n
\mathcal{C}_{n-1}$. Then, the equality $A\mathcal{S}_n\Phi_{n-1} = A\Phi_n = \mathcal{C}_n\Psi_{n+r} = \mathcal{C}_n\mathcal{T}_{n+r}\Psi_{n+r-1} =
\mathcal{S}_n\mathcal{C}_{n-1}\Psi_{n+r-1}$ shows that $A\Phi_{n-1} = \mathcal{C}_{n-1}\Psi_{n+r-1}$. Iterating this procedure we obtain (ii).

Combining (\ref{GPRMF}) and recurrence relations (\ref{TM}),
$$
A \Phi_j = \mathcal{C}_j \Psi_{j+r} = \mathcal{C}_j \mathcal{T}_{j+r} \Psi_{j+r-1},
\qquad
A \Phi_j = A \mathcal{S}_j \Phi_{j-1} = \mathcal{S}_j  \mathcal{C}_{j-1} \Psi_{j+r-1}.
$$
Therefore, $\mathcal{C}_j \mathcal{T}_{j+r} \Psi_{j+r-1} = \mathcal{S}_j \mathcal{C}_{j-1} \Psi_{j+r-1}$, or equivalently
$$
\mathcal{C}_j \mathcal{B}_{j+r} \pmatrix{z\psi_{j+r-1} \cr \psi_{j+r-1}^*} =
\mathcal{A}_j \tilde{\mathcal{C}}_{j-1} \pmatrix{z\psi_{j+r-1} \cr \psi_{j+r-1}^*}.
$$
Taking into account that $z\psi_j$, $\psi_j^*$ are relatively prime and $\deg C_j=r$,  $\deg D_j \leq r-1$, relations (\ref{MAT}) and (\ref{MATILDA})
follow from Lemma \ref{MATRICES}.

To prove (\ref{DETAC}) notice that $A= C_0 \psi_r + D_0 \psi_r^* = C_0^* \psi_r^* + zD_0^* \psi_r$, hence we have the equality
$(C_0-zD_0^*)\psi_r=(C_0^*-D_0)\psi_r^*$. Since $\psi_r$, $\psi_r^*$  are relatively prime this implies $C_0(0) \psi_r = C_0^* - D_0$ and $C_0(0)
\psi_r^* = C_0 - z D_0^*$. So,
$$
C_0(0) A = C_0(0) (C_0 \psi_r + D_0 \psi_r^*) = C_0C_0^* - z D_0D_0^* = \det\mathcal{C}_0.
$$
Besides, from (\ref{MAT}) we find that $\det\mathcal{C}_j \propto \det\mathcal{C}_0$ for $j=1,\dots,n$. Evaluating at $z=0$ we finally obtain
$\det\mathcal{C}_j = (C_j(0)/C_0(0)) \det\mathcal{C}_0 = C_j(0) A$.
\end{proof}

The equivalence (i) $\Leftrightarrow$ (iii) of the previous theorem means that the last condition ($j=n$) in (\ref{GPOLO}) or (\ref{GPOLO1}) suffices
for the equivalence in Theorem \ref{GENERALPROBLEMTH}.

There exist also inverse relations between the finite segments of MOP $\bigl(\varphi_j\bigr)_{j=0}^n$ and $\bigl(\psi_j\bigr)_{j=0}^{n+r}$. The
polynomial matrix coefficients of these inverse relations are not independent of the polynomial matrix coefficients $\mathcal{C}_j$ of the direct
relations. Indeed, both polynomial matrix coefficients are essentially adjoints of each other, understanding the adjoint of a $2\times2$ matrix
$M=\scriptsize\pmatrix{M_1&M_2 \cr M_3&M_4}$ as the matrix $\adj(M)=\scriptsize\pmatrix{M_4&-M_2 \cr -M_3&M_1}$. Thus, given a $2\times2$ polynomial
matrix $M$ in $\PP_r$, $\adj(M)$ is a $2\times2$ polynomial matrix in $\PP_r$ satisfying
$$
\adj(M) \, M = (\det M) \, I,
$$
where $I$ is the identity matrix of the same size as $M$.

\begin{thm} \label{EQGENERALPROBLEMTHMF}

If $u$, $v$ are quasi-definite in $\PP_n$, $\PP_{n+r}$ respectively, the following statements are equivalent:
\begin{itemize}
\item[(i)] $u \equiv vL$ in $\PP_n$  for some hermitian Laurent polynomial $L$ of degree $r$.
\item[(ii)] There exist $\mathcal{X}_r,\dots,\mathcal{X}_n\in\J_r^\reg$ such that
\begin{equation} \label{BOLI}
\Psi_{j+r} = \mathcal{X}_j \Phi_j, \qquad j=r,\dots,n.
\end{equation}
\item[(iii)] There exists $\mathcal{X}_n\in\J_r^\reg$ such that
\begin{equation} \label{BOLI2}
\Psi_{n+r} = \mathcal{X}_n \Phi_n.
\end{equation}
\end{itemize}
The matrices $\mathcal{X}_j$ are the only solutions of $(\ref{BOLI})$ in $\J_r$, so $\mathcal{X}_r$ is determined by
\begin{equation} \label{IC-INV}
\mathcal{X}_r \Phi_r = \Psi_{2r}, \qquad \mathcal{X}_r\in\J_r.
\end{equation}
Besides, we have the relations
\begin{equation} \label{SHP}
\mathcal{T}_{j+r} \mathcal{X}_{j-1} = \mathcal{X}_j \mathcal{S}_j, \qquad j = r+1,\dots,n.
\end{equation}
\begin{equation} \label{SHP2}
\mathcal{B}_{j+r} \tilde\mathcal{X}_{j-1} = \mathcal{X}_j \mathcal{A}_j, \qquad j = r+1,\dots,n.
\end{equation}
\begin{equation} \label{detX}
\det\mathcal{X}_j \propto A, \qquad j=r,\dots,n.
\end{equation}
\begin{equation} \label{MIP}
\mathcal{C}_j \mathcal{X}_j = A \pmatrix{1&0 \cr 0&1}, \qquad  j = r,\dots,n.
\end{equation}
\begin{equation} \label{X-C}
\mathcal{X}_j = {1\over C_j(0)} \, \adj (\mathcal{C}_j), \qquad j=r,\cdots,n.
\end{equation}

\end{thm}

\begin{proof}
If $u \equiv vL$ in $\PP_n$, Theorem \ref{DIRECTPROBLEMTHMF} assures the existence of $\mathcal{C}_j\in\J_r^\reg$ such that $A \Phi_j = \mathcal{C}_j
\Psi_{j+r}$ for $j=0,\cdots,n$. Multiplying this identity on the left by $\adj (\mathcal{C}_j)$ and taking into account (\ref{DETAC}) we find that
$\Psi_{j+r} = \mathcal{X}_j \Phi_j$ for $j=0,\cdots,n$, where $\mathcal{X}_j = \adj (\mathcal{C}_j)/C_j(0)\in\J_r^\reg$. Then, (\ref{SHP}),
(\ref{SHP2}), (\ref{detX}) and (\ref{MIP}) are a direct consequence of (\ref{MAT}), (\ref{MATILDA}) and (\ref{DETAC}). The uniqueness of
$\mathcal{X}_j\in\J_r$ for $j \geq r$ follows from Corollary \ref{BASISC} and the fact that (\ref{BOLI}) is equivalent to $\psi_{j+r} = X_j \varphi_j
+ Y_j \varphi_j^*$, where $X_j\in\PP_r$, $Y_j\in\PP_{r-1}$ are the polynomials appearing in $\mathcal{X}_j = {\scriptsize \pmatrix{X_j & Y_j \cr
zY_j^* & X_j^*}}$.

It only remains to prove (iii) $\Rightarrow$ (i). Multiplying (\ref{BOLI2}) on the left by $\mathcal{C}_n=\adj(\mathcal{X}_n)\in\J_r^\reg$ we obtain
$A\Phi_n=\mathcal{C}_n\Psi_{n+r}$ where $A=\det\mathcal{X}_n$ is a self-reciprocal polynomial of degree $2r$. This proves that $u \equiv vAz^{-r}$
due to Theorem \ref{DIRECTPROBLEMTHMF}.
\end{proof}

Concerning the polynomial matrix coefficients $\mathcal{X}_j\in\J_r^\reg$ of the inverse relations, when it is necessary we will use the explicit
notation
$$
\mathcal{X}_j = \pmatrix{X_j & Y_j \cr zY_j^* & X_j^*}, \quad \deg X_j = r, \quad \deg Y_j = r-1,
$$
so that $\Psi_{j+r}=\mathcal{X}_j\Phi_j$ is equivalent to $\psi_{j+r}=X_j\varphi_j+Y_j\varphi_j^*$. This shows that $X_j$ is monic. Besides, from
(\ref{X-C}) we have the relations $X_j=C_j^*/C_j(0)$, $Y_j=-D_j/C_j(0)$.

The proof of the previous theorem shows that, when $u\equiv vL$ in $\PP_n$ for some hermitian Laurent polynomial $L$ of degree $r$,
\begin{equation} \label{EQUIVFS}
\Psi_{j+r} = \mathcal{X}_j \Phi_j, \qquad \mathcal{X}_j\in\J_r, \qquad j=0,\dots,n,
\end{equation}
and not only for $j \geq r$. Indeed, the proof of the theorem implies that (\ref{EQUIVFS}) has solutions $\mathcal{X}_j\in\J_r^\reg$ for $j<r$ too.
The only difference is that, contrary to $j \geq r$, (\ref{EQUIVFS}) does not determine $\mathcal{X}_j$ univocally for $j<r$, as Corollary
\ref{BASISC} points out. The reason is that $\mathfrak{B}=\{z^k\psi_j\}_{k=0}^r \cup \{z^k\psi_j^*\}_{k=0}^{r-1}$ is linearly independent for $j \geq
r$, but not for $j<r$. Actually, when $j<r$, Lemma \ref{BASISOC} shows that $\rank(\mathfrak{B})=j+r+1$, so the solutions $\mathcal{X}_j$ of
(\ref{EQUIVFS}) form an affine subspace of dimension $r-j$.

Among the solutions of (\ref{EQUIVFS}) for $j<r$ there is a choice of special interest: similar arguments to those at the beginning of the proof of
Theorem \ref{DIRECTPROBLEMTHMF} show that Lemma \ref{CT=SC} (i), together with (\ref{IC-INV}), assures that $(\ref{SHP})$ can be extended in a unique
way to $j=1,\dots,r$, giving rise to particular solutions $\mathcal{X} _0,\dots,\mathcal{X}_{r-1}\in\J_r^\reg$ of (\ref{EQUIVFS}). The choice of
$\cX_j$ determined by the extension of (\ref{SHP}) has the particularity that $\det\mathcal{X}_j$ is independent of $j$ up to numerical factors.
Indeed, this property characterizes such a particular choice because different solutions of (\ref{EQUIVFS}) can not have proportional determinants:
let $\cX^{(1)},\cX^{(2)}\in\J_r$ be such that $\Psi_{j+r}=\cX^{(k)}\Phi_j$. Then, $(\det\cX^{(k)})\Phi_j=\cC^{(k)}\Psi_{j+r}$ with
$\cC^{(k)}=\adj(\cX^{(k)})$. If $\det\cX^{(2)}=\lambda\det\cX^{(1)}$, $\lambda\in\R^*$, Lemma \ref{MATRICES} assures that
$\cC^{(2)}=\lambda\cC^{(1)}$, thus $\det\cX^{(2)}=\lambda^2\det\cX^{(1)}$, which implies $\lambda=1$, so $\cX^{(2)}=\cX^{(1)}$.

Properties (\ref{MAT}) and (\ref{SHP}) are the cornerstone of the main objective of this section: a new characterization of the relation $u=vL$ in
terms of a recurrence for the corresponding Schur parameters. Like in the previous characterizations, the $J$-self-reciprocal matrices play an
important role, but now only one MOP of $u$ and $v$ enters in the equivalence, and it appears only in the initial condition for the recurrence. The
direct and inverse relations between the MOP of $u$ and $v$ lead to different characterizations, depending on whether the hermitian Laurent
polynomial $L$ is fixed or not. Indeed, $L$ appears explicitly only in the initial condition for the direct characterization.

\begin{thm} \label{NEW}

Let $u$, $v$ be quasi-definite in $\PP_n$, $\PP_{n+r}$ respectively and consider an index $m\in\{0,\dots,n\}$.
\begin{itemize}
\item[(i)] Given a hermitian Laurent polynomial $L$ of degree $r$, $u \equiv vL$ in $\PP_n$ iff there exist $\mathcal{C}_m\in\J_r^\reg$ and
$\mathcal{C}_{m+1},\dots,\mathcal{C}_n\in\J_r$ such that
$$
\begin{array}{l}
\mathcal{C}_m \Psi_{m+r} = A \Phi_m, \kern111pt \mbox{\rm (Direct Initial Condition)}
\medskip \\
\mathcal{C}_j \mathcal{T}_{j+r} = \mathcal{S}_j \mathcal{C}_{j-1}, \quad j=m+1,\dots,n. \kern20pt \mbox{\rm (Direct Recurrence)}
\end{array}
$$
Moreover, $A\Phi_j=\mathcal{C}_j\Psi_{j+r}$, $\mathcal{C}_j\in\J_r^\reg$ and $\det\mathcal{C}_j\propto A$ for $j=m,\dots,n$.
\item[(ii)] There is a hermitian Laurent polynomial $L$ of degree $r$ such that $u \equiv vL$ in $\PP_n$ iff there exist
$\mathcal{X}_m\in\J_r^\reg$ and $\mathcal{X}_{m+1},\dots,\mathcal{X}_n\in\J_r$ such that
$$
\begin{array}{l}
\mathcal{X}_m \Phi_m = \Psi_{m+r}, \kern122pt \mbox{\rm (Inverse Initial Condition)}
\medskip \\
\mathcal{T}_{j+r} \mathcal{X}_{j-1} = \mathcal{X}_j \mathcal{S}_j, \quad j=m+1,\dots,n. \kern20pt \mbox{\rm (Inverse Recurrence)}
\end{array}
$$
Moreover, $\Psi_{j+r}=\mathcal{X}_j\Phi_j$, $\mathcal{X}_j\in\J_r^\reg$ and $\det\mathcal{X}_j \propto A$ for $j=m,\dots,n$.
\end{itemize}

\end{thm}

\begin{proof}
We will prove only (i), the proof of (ii) being similar. In view of Theorem \ref{DIRECTPROBLEMTHMF}, it suffices to show that Direct Initial Condition
and Direct Recurrence imply $A\Phi_j=\mathcal{C}_j\Psi_{j+r}$, $j=m,\dots,n$ and $\mathcal{C}_j\in\J_r^\reg$, $j=m+1,\dots,n$ when
$\mathcal{C}_m\in\J_r^\reg$. Direct Recurrence yields $(1-|b_{j+r}|^2)\det\mathcal{C}_j=(1-|a_j|^2)\det\mathcal{C}_{j-1}$, thus
$\mathcal{C}_m\in\J_r^\reg$ implies $\mathcal{C}_j\in\J_r^\reg$ for $j=m+1,\dots,n$. Also, Direct Recurrence and Direct Initial Condition combined
with recurrence relations (\ref{TM}) lead to $A\Phi_j = A\mathcal{S}_j\cdots\mathcal{S}_{m+1}\Phi_m =
\mathcal{S}_j\cdots\mathcal{S}_{m+1}\mathcal{C}_m\Psi_{m+r} = \mathcal{C}_j\mathcal{T}_{j+r}\cdots\mathcal{T}_{m+r+1}\Psi_{m+r} =
\mathcal{C}_j\Psi_{j+r}$ for $j=m,\dots,n$.
\end{proof}

Some special cases of the above theorem will be of interest for us. We will summarize them.

\begin{thm} \label{EQUIVALENCESRRGC}

Let $u$, $v$ be quasi-definite in $\PP_n$, $\PP_{n+r}$ respectively.

\item{\underline{\rm Direct characterization}} Given a hermitian Laurent polynomial $L$ of degree $r$, the following statements are equivalent:
\begin{itemize}
\item[(i)] $u \equiv vL$ in $\PP_n$.
\item[(ii)] There exist $\mathcal{C}_0\in\J_r^\reg$ and $\mathcal{C}_1,\dots,\mathcal{C}_n\in\J_r$ such that
$$
\begin{array}{l}
\mathcal{C}_0\Psi_r = A \pmatrix{1 \cr 1}, \kern98pt \mbox{\rm (Initial Condition D)}
\medskip \\
\mathcal{C}_j \mathcal{T}_{j+r} = \mathcal{S}_j \mathcal{C}_{j-1}, \quad j=1,\dots,n. \kern25pt \mbox{\rm (Recurrence D)}
\end{array}
$$
\end{itemize}

\item{\underline{\rm Inverse characterization}} The following statements are equivalent:
\begin{itemize}
\item[(i)] $u \equiv vL$ in $\PP_n$ for some hermitian Laurent polynomial $L$ of degree $r$.
\item[(ii)] There exist $\mathcal{X}_r\in\J_r^\reg$ and $\mathcal{X}_{r+1},\dots,\mathcal{X}_n\in\J_r$ such that
$$
\begin{array}{l}
\mathcal{X}_r \Phi_r = \Psi_{2r}, \kern139pt \mbox{\rm (Initial Condition I1)}
\medskip \\
\mathcal{T}_{j+r} \mathcal{X}_{j-1} = \mathcal{X}_j \mathcal{S}_j, \quad j=r+1,\dots,n. \kern25pt \mbox{\rm (Recurrence I1)}
\end{array}
$$
\item[(iii)] There exist $\mathcal{X}_0\in\J_r^\reg$ and $\mathcal{X}_1,\dots,\mathcal{X}_n\in\J_r$ such that
$$
\begin{array}{l}
\mathcal{X}_0 \pmatrix{1 \cr 1} = \Psi_r, \kern124pt \mbox{\rm (Initial Condition I2)}
\medskip \\
\mathcal{T}_{j+r} \mathcal{X}_{j-1} = \mathcal{X}_j \mathcal{S}_j, \quad j=1,\dots,n. \kern39pt \mbox{\rm (Recurrence I2)}
\end{array}
$$
\end{itemize}

\end{thm}

The difference between the inverse characterizations I1 and I2 is that the initial condition determines univocally the initial matrix $\mathcal{X}_r$
for I1 but not the initial matrix $\mathcal{X}_0$ for I2, thus there is a freedom in such initial matrix for I2. We will go back to this point later
on.

Theorems \ref{DIRECTPROBLEMTHMF}, \ref{EQGENERALPROBLEMTHMF}, \ref{NEW} and \ref{EQUIVALENCESRRGC} have an obvious generalization to the
quasi-definite case.

Theorem \ref{EQUIVALENCESRRGC} shows that the regularity of $\mathcal{C}_j$, $j\neq0$, and $\mathcal{X}_j$, $j\neq r$, is a superfluous condition in
statement (ii) of Theorems \ref{DIRECTPROBLEMTHMF} and \ref{EQGENERALPROBLEMTHMF} respectively. Remember that the regularity of $\mathcal{C}_0$ is
equivalent to $v[L]\neq0$. On the other hand, the regularity conditions for $\mathcal{X}_j$ in Theorems \ref{EQGENERALPROBLEMTHMF} and
\ref{EQUIVALENCESRRGC} can be completely avoided if we do not fix the degree of $L$. In other words, if $\mathcal{X}_j\in\J_r\setminus\J_r^\reg$ then
$u \equiv vL$ in $\PP_n$ too, but $\deg L<r$, as follows from the following proposition.

\begin{prop} \label{RED}

If $\Psi_{j+r}=\mathcal{X}_j\Phi_j$ with $\mathcal{X}_j\in\J_r\setminus\J_r^\reg$, then $\Psi_{j+r-1}=\hat\mathcal{X}_j\Phi_j$ with
$\hat\mathcal{X}_j\in\J_{r-1}$.

\end{prop}

\begin{proof}
Suppose $\Psi_{j+r}=\mathcal{X}_j\Phi_j$, $\mathcal{X}_j\in\J_r$ with $X_j(0)=0$. Then $b_{j+r}=Y_j(0)$ and $X_j = z \hat X_j$ with $\hat X_j$ monic
of degree $r-1$. Thus we can write
$$
\Psi_{j+r} = \pmatrix{z \hat X_j & Y_j \cr zY_j^* & \hat X_j^*} \Phi_j.
$$
From $\Psi_{j+r}=\mathcal{T}_{j+r}\Psi_{j+r-1}$ we get $\Psi_{j+r-1}=\hat\mathcal{X}_j\Phi_j$ where
$$
\hat\mathcal{X}_j = {1\over 1-|b_{j+r}|^2}
\pmatrix{\hat{X}_j-b_{j+r}Y_j^* & z^{-1}(Y_j-b_{j+r}\hat{X}_j^*) \cr z(Y_j^*-\overline{b}_{j+r}\hat{X}_j) & \hat{X}_j^*-\overline{b}_{j+r}Y_j}.
$$
Since $Y_j(0)-b_{j+r}\hat{X}_j^*(0)=0$ and $\hat{X}_j^*(0)-\overline{b}_{j+r}Y_j(0)=1-|b_{j+r}|^2\neq0$ we conclude that
$\hat\mathcal{X}_j\in\J_{r-1}$.
\end{proof}

\section{Direct and inverse problems} \label{DIP}

In the previous section, given two hermitian linear functionals $u$, $v$ and a hermitian Laurent polynomial $L$, we have studied the relation
$u\equiv vL$ obtaining characterizations in terms of linear relations with polynomial coefficients between the corresponding MOP, as well as in terms
of a matrix difference equation between the related Schur parameters. In this section we will use these results to answer the following question:
Which conditions ensure the quasi-definiteness of $u=vL$ or $v$ once we know that the other functional is quasi-definite?

Indeed we will answer this question in the more general context of quasi-definite functionals in some subspace $\PP_n$: we will try to know the
minimum length of the finite segments of MOP for one of the functionals assuming that the other functional has a finite segment of MOP with a given
length. Like in the previous section, the main goal is to develop techniques for this problem based almost exclusively on the knowledge of the Schur
parameters.

The new results will seem quite similar to those of the previous section, however they provide new information: in the previous section we assumed
that $u$ and $v$ had finite segments of MOP of certain length and we asked about a characterization of the relation $u \equiv vL$ in some subspace
$\PP_n$; now we will consider the relation $u = vL$ as a data and we will ask about the length of the finite segments of MOP.

\subsection{Direct problem}

The direct problem refers to the case where we suppose that a hermitian functional $v$ with a finite segment of MOP $(\psi_j)_{j=0}^m$ and a
hermitian polynomial $L$ of degree $r$ are given. Then, we will try to obtain information about the functional $u=vL$ and its finite segments of MOP
$(\varphi_j)_{j=0}^n$. Our first result is essentially a reinterpretation of relation (\ref{GPRMF}).

\begin{thm} \label{DIRECTPROBLEMTH}

Let $v$ be quasi-definite in $\PP_{n+r}$ and let $L$ be a hermitian Laurent polynomial of degree $r$. Then, $u=vL$ is quasi-definite in $\PP_n$ iff
there exists $\mathcal{C}_j\in\J_r^\reg$ such that $A$ divides $\mathcal{C}_j \Psi_{j+r}$ for $j=0,\dots,n$.

Besides, $\det\mathcal{C}_j \propto A$ and there is a unique choice of $\mathcal{C}_j$ such that $C_j^*(0)=A(0)$. For such a choice the finite
segment of MOP with respect to $u$ is given by $A \Phi_j = \mathcal{C}_j \Psi_{j+r}$ for $j=0,\dots,n$.

\end{thm}

\begin{proof}

First of all notice that, no matter the value of $\lambda_j\in\C^*$, $A$ divides $\mathcal{C}_j\Psi_{j+r}$ iff it divides
$\hat\mathcal{C}_j\Psi_{j+r}$ with $\hat\mathcal{C}_j=\scriptsize\pmatrix{\lambda_j&0 \cr 0&\overline\lambda_j}\mathcal{C}_j$, and
$\mathcal{C}_j\in\J_r^\reg$ iff $\hat\mathcal{C}_j\in\J_r^\reg$. Therefore, we can suppose without loss of generality that $C_j^*(0)=A(0)$. Then, the
divisibility condition is equivalent to $A\Phi_j=\mathcal{C}_j\Psi_{j+r}$ with $\varphi_j$ a monic polynomial of degree $j$ which, for the moment,
has no relation with $u$. Taking into account Theorem \ref{DIRECTPROBLEMTHMF}, to prove the result we only need to see that $\varphi_j$ is the $j$-th
MOP with respect to $u$. This follows from the orthogonality conditions of $\psi_{j+r}$ with respect to $v$, which give
$$
\begin{array}{l}
u\bigl[\varphi_j  z^{-k}\bigl] = v\bigl[\bigl(C_j \psi_{j+r} + D_j \psi_{j+r}^*\bigr) z^{-(k+r)}\bigr] = 0, \qquad 0\leq k\leq j-1,
\medskip \\
u\bigl[\varphi_j z^{-j}\bigl] = v\bigl[\bigl(C_j \psi_{j+r} + D_j \psi_{j+r}^*\bigr) z^{-(j+r)}\bigr] = C_j(0)\varepsilon_{j+r}\not= 0.
\end{array}
$$
The rest of the theorem is a consequence of Theorem \ref{DIRECTPROBLEMTHMF}.
\end{proof}

The above results allow us to obtain a necessary and sufficient condition for the quasi-definiteness of the functional $u=vL$ in terms of
determinants involving the MOP of $v$.

\begin{prop} \label{DETERMINANT}

Let $v$ be quasi-definite in $\PP_{n+r}$ and let $L$ be a hermitian Laurent polynomial of degree $r$. Then, $u=vL$ is quasi-definite in $\PP_n$ iff
$\det M^{(m)}\neq0$ for $m=0,\dots,n+1$, where $M^{(m)}=(M^{(m)}_{ij})_{i,j=1}^{2r}$ is the square matrix of order $2r$ given by
$$
M^{(m)}_{ij} = \cases{\bigl(z^{j-1}\psi_{m+r}\bigr)^{(l_i}(\zeta_i), & $j=1,\dots r$,
                      \cr \bigl(z^{j-r-1} \psi_{m+r}^*\bigr)^{(l_i}(\zeta_i), & $j=r+1,\dots,2r$,}
                      \quad i=1,\dots,2r,
$$
with $\zeta_1,\dots,\zeta_{2r}$ the roots of $A$ counting the multiplicity and $l_i$ the number of roots $\zeta_j$,  $j<i$, such that $\zeta_j =
\zeta_i$.

\end{prop}

\begin{proof}
By Theorem \ref{DIRECTPROBLEMTH}, to decide the quasi-definiteness of $u=vL$ in $\PP_n$, we simply have to analyze the existence of unique
polynomials $C_m$, $D_m$ with $\deg C_m=r$, $\deg D_m\leq r-1$, $C_m(0)\neq0$, $C_m^*(0)=A(0)$, such that $A$ divides $C_m\psi_{m+r}+D_m\psi_{m+r}^*$
for $m=0,\dots,n$.

Let us write $C_m(z)  = \sum_{k=0}^r c_{m,k} z^k$ and $D_m(z) = \sum_{k=0}^{r-1}d_{m,k}z^k$. The condition $C_m^*(0)=A(0)$ only means that $c_{m,r}$
is the leading coefficient of $A$. Then, the existence of unique polynomials $C_m$, $D_m$ is equivalent to the existence and uniqueness of the $2r$
coefficients $c_{m,0},\dots,c_{m,r-1}$ and $d_{m,0},\dots,d_{m,r-1}$, while the condition $C_m(0)\neq0$ becomes $c_{m,0}\neq0$.

If $\zeta_1,\dots,\zeta_{2r}$ denote the $2r$ roots of the polynomial $A$ counting the multiplicity and $l_i$ is the number of roots $\zeta_j$ such
that $\zeta_j=\zeta_i$ for $j<i$, the divisibility condition is equivalent to the system
$$
(C_m\psi_{m+r})^{(l_i}(\zeta_i) +  (D_m \psi_{m+r}^*)^{(l_i}(\zeta_i) = 0, \qquad i=1,\dots 2r.
$$
This system has a unique solution in $c_{m,k},$  $d_{m,k},$ $k=0,\dots,r-1$, exactly when $\det M^{(m)} \neq 0$.

It remains to translate the condition $c_{m,0}\neq0$. The solution for $c_{m,0}$ is proportional to the determinant of a matrix $M$ obtained
substituting in $M^{(m)}$ the first column $(\psi_{m+r}^{(l_i}(\zeta_i))_{i=1}^{2r}$ by $(z^r\psi_{m+r}^{(l_i}(\zeta_i))_{i=1}^{2r}$. Since
(\ref{RR}) and (\ref{RRR}) imply that $\spn\{z^{j+1}\psi_{m+r},z^j\psi_{m+r}^*\}=\spn\{z^j\psi_{m+r+1},z^j\psi_{m+r+1}^*\}$, we see that $\det M$
vanishes at the same time than $\det M^{(m+1)}$. Hence, $c_{m,0}\neq0$ is equivalent to $\det M^{(m+1)} \neq 0$.
\end{proof}

The condition given by the above proposition is theoretically interesting but in practice it is not manageable, specially for polynomial
perturbations of high degree $r$ due to the need to evaluate determinants of $2r \times 2r$ matrices. Even in case of low degree $r$, the practical
application of the previous result needs the construction of the MOP $\psi_j$ and the evaluation at some points of these MOP and their derivatives.

When $r=1$ the self-reciprocal polynomial $A$ has two roots $\zeta_1$, $\zeta_2$ such that $\zeta_2=1/\overline\zeta_1$ or $\zeta_1,\zeta_2\in\T$,
$\zeta_1\neq\zeta_2$. Obviously, when $v$ is positive definite and $\zeta_2=1/\overline\zeta_1$ the functional $vL$ is positive definite too.
However, in general, $v$ quasi-definite in $\PP_{n+r}$ implies $vL$ quasi-definite in $\PP_n$ iff (see \cite{Su93,Ca,DaHeMa07,ACMV})
$K_m(\zeta_1,1/\overline\zeta_2)\neq0$ for $m=1,\dots,n+1$, where $K_m(z,w)=\sum_{j=0}^m\varepsilon_j^{-1}\psi_j(z)\overline{\psi_j(w)}$ is the
$m$-th kernel associated with the MOP $(\psi_j)$.

Nevertheless, it is naive to think that the general situation can be solved by factoring the polynomial $A$. Consider for instance a positive
definite functional $v$ and let $A(z)\propto(z-\zeta_1)(z-\zeta_2)$ with $\zeta_1,\zeta_2\in\T$, $\zeta_1 \neq \zeta_2$, satisfying
$K_m(\zeta_1,1/\overline\zeta_2)=0$ for some $m$. Then $vL$ is not quasi-definite but $vL^2$ is positive definite.

A more practical characterization of the quasi-definiteness of $u=vL$, which avoids the construction of the MOP of $v$ and does not need the
calculation of determinants, is given in terms of the recurrence for the Schur parameters.

\begin{thm} \label{MATRIXCOEFFPD}

Let $v$ be quasi-definite in $\PP_{n+r}$ and let $L$ be a hermitian Laurent polynomial of degree $r$. Then, $u=vL$ is quasi-definite in $\PP_n$ iff
there exist $a_1,\dots,a_n\in\C$ and $\mathcal{C}_0,\dots,\mathcal{C}_n\in\J_r^\reg$ such that
\begin{equation} \label{DMATINITIALCONDITION}
\kern-85pt \mathcal{C}_0 \Psi_r = A \pmatrix{1 \cr 1},
\end{equation}
\begin{equation} \label{DMAT}
\mathcal{C}_j \mathcal{T}_{j+r} = \mathcal{S}_j \mathcal{C}_{j-1}, \qquad j=1,\dots,n.
\end{equation}
Besides, $A\Phi_j=\mathcal{C}_j\Psi_{j+r}$, $\det\mathcal{C}_j \propto A$, $j=0,\dots,n$, and $a_j=\varphi_j(0)\in\C\setminus\T$, $j=1,\dots,n$.

\end{thm}

\begin{proof}
In view of Theorem \ref{EQUIVALENCESRRGC}, we only need to prove that $u$ is quasi-definite in $\PP_n$ when (\ref{DMATINITIALCONDITION}) and
(\ref{DMAT}) hold. Define $\Phi_j= \mathcal{S}_j\cdots\mathcal{S}_1\scriptsize\pmatrix{1 \cr 1}$. Then, (\ref{DMATINITIALCONDITION}), (\ref{DMAT})
and the recurrence relation for $(\Psi_j)_{j=0}^{n+r}$ yield for $j=0,\dots,n$,
$$
A\Phi_j = A\mathcal{S}_j\cdots\mathcal{S}_1\pmatrix{1 \cr 1} = \mathcal{S}_j\cdots\mathcal{S}_1\mathcal{C}_0\Psi_r =
\mathcal{C}_j\mathcal{T}_{j+r}\cdots\mathcal{T}_{r+1}\Psi_r = \mathcal{C}_j\Psi_{j+r}.
$$
Therefore, Theorem \ref{DIRECTPROBLEMTH} shows that $u$ is quasi-definite in $\PP_n$.
\end{proof}

The above results yield a direct relation between the Schur parameters of $u=vL$ and $v$, which can be obtained setting $z=0$ in the equivalent
version $\mathcal{C}_j \mathcal{B}_{j+r} = \mathcal{A}_j \tilde{\mathcal{C}}_{j-1}$ of (\ref{DMAT}) and using $C_j^*(0)=A(0)$.

\begin{cor} \label{SCHURPARAMETERSDP}
If $\alpha$ is the leading coefficient of $A$, the $j$-th Schur parameter $a_j$ of $u=vL$ can be obtained from the $j+r$-th Schur parameter $b_{j+r}$
of $v$ by
\begin{equation} \label{FSHURPARAMETERSPD}
a_j = {\alpha b_{j+r} -  \overline{D_{j-1}^*(0)} \over  C_{j-1}(0)}.
\end{equation}

\end{cor}

Theorem \ref{MATRIXCOEFFPD} and Corollary \ref{SCHURPARAMETERSDP} provide an algorithm to obtain the Schur parameters $(a_j)$ of $u=vL$ from the
Schur parameters $(b_j)$ of $v$.

\bigskip

\noindent{\bf Algorithm D}
\begin{itemize}
\item Determination of $\mathcal{C}_0\in\J_r$ from initial condition (\ref{DMATINITIALCONDITION}) and $\Psi_r$, $A$.
\item For $j=1,2,\dots$
    \begin{itemize}
    \item[$\bullet$] While $C_{j-1}(0)\neq 0$, calculation of  $a_j$ from (\ref{FSHURPARAMETERSPD}) and $b_{j+r}$, $\mathcal{C}_{j-1}$.
    \item[$\bullet$] Determination of $\mathcal{C}_j\in\J_r$ from recurrence (\ref{DMAT}) and $a_j$, $b_{j+r}$, $\mathcal{C}_{j-1}$.
    \end{itemize}
\end{itemize}

The fact that the $j$-th step of the above algorithm actually gives a matrix $\mathcal{C}_j\in\J_r$ is a consequence of Lemma \ref{CT=SC} (ii) and
the equivalence between $C_{j-1}(0)\neq0$ and $\mathcal{C}_{j-1}\in\J_r^\reg$ when $\mathcal{C}_{j-1}\in\J_r$.

In short, the fact that Algorithm D works from $j=1$ to $j=n$ will be called the {\it $n$-consistence} of recurrence (\ref{DMAT}). We will say that
the recurrence is {\it consistent} if it works for any $j\geq 1$. Of course, this is an abuse of language because the consistence depends, not only
on recurrence (\ref{DMAT}), but also on initial condition (\ref{DMATINITIALCONDITION}).

The consistence relies on the fact that $C_j(0)\neq 0$ at each step. Suppose that the recurrence fails at the $(n+1)$-th step, i.e., it is
$n$-consistent and not $(n+1)$-consistent. Then $C_{n-1}(0) \neq 0$ and $C_n(0) =0$, that is, $\mathcal{C}_{n-1}\in\J_r^\reg$ but
$\mathcal{C}_n\in\J_r\setminus\J_r^\reg$. Recurrence (\ref{DMAT}) shows that this is equivalent to $|a_{n-1}|\neq1$ and $|a_n|=1.$ So, the
$n$-consistence condition can be written as $|a_j|\neq 1$ for $j=1,\dots,n-1$, which means that $u=vL$ has a finite segment of MOP of length $n$,
i.e., it is quasi-definite in $\PP_{n-1}$.

Contrary to Theorem \ref{DETERMINANT}, Algorithm D only requires the knowledge of the Schur parameters of $v$ and a single MOP $\psi_r$ with the same
degree $r$ as the polynomial perturbation $L$. Furthermore, this algorithm makes the calculation of determinants completely unnecessary. As an
example, we will develop explicitly Algorithm D for $r=1$.

\subsubsection{The case $r=1$}

Consider a hermitian functional $v$ with MOP $(\psi_j)$ and a hermitian Laurent polynomial $L$ of degree 1. We can write $L=P+P_*$, $P(z)=\alpha
z+\beta$, $\alpha\in\C^*$, $\beta\in\R$, so $A(z)=zL(z)=\alpha z^2 + 2\beta z + \overline\alpha$. The MOP $(\varphi_j)$ of the modified functional
$u=vL$, if they exist, are given by
$$
A\varphi_j=(\alpha z+c_j)\psi_{n+1}+d_j\psi_{n+1}^*,
$$
for some $c_j\in\R$, $d_j\in\C$. This relation and its reversed can be combined in
$$
A\Phi_j=\mathcal{C}_j\Psi_{j+1}, \qquad \mathcal{C}_j=\pmatrix{\alpha z+c_j & d_j \cr \overline d_j z & \overline\alpha+c_jz}.
$$

Also, recurrence (\ref{DMAT}) becomes
$$
\cases{
c_{j-1} + \overline{d}_{j-1} a_j = c_j + d_j \overline{b}_{j+1},
\cr
\overline{\alpha} a_j = c_j b_{j+1} + d_j,
\cr
c_{j-1} a_j + d_{j-1} = \alpha b_{j+1},
}
$$
which can be written as
\begin{equation} \label{RR-D-1}
a_j = \frac{\alpha b_{j+1}-d_{j-1}}{c_{j-1}},
\quad
\pmatrix{1&\overline{b}_{j+1} \cr b_{j+1}&1} \pmatrix{c_j \cr d_j} = \pmatrix{c_{j-1}+\overline{d}_{j-1}a_j \cr \overline{\alpha}a_j}.
\end{equation}

On the other hand, initial condition (\ref{DMATINITIALCONDITION}) is equivalent to $\alpha z^2 + 2\beta z +  \overline{\alpha} = (z+c_1)(z+b_1) +
d_0(\overline{b}_1 z + 1)$, i.e.,
\begin{equation} \label{IC-D-1}
\pmatrix{1&\overline{b}_1 \cr b_1& 1} \pmatrix{c_0 \cr d_0\cr} = \pmatrix{2\beta - \alpha b_1 \cr 1}.
\end{equation}
This provides unique $c_0$, $d_0$ for any $P$ and any possible value of $b_1\in\C\setminus\T$.

Finally, Algorithm D can be explicitly formulated in the following way:
\begin{itemize}
\item Calculation of $c_0$, $d_0$ from $P$, $b_1$ using (\ref{IC-D-1}).
\item For $j=1,2,\dots,$ while $c_{j-1}\neq0$, calculation of $a_j$, $c_j$, $d_j$ from $b_{j+1}$, $c_{j-1}$, $d_{j-1}$ using (\ref{RR-D-1}).
\end{itemize}
This algorithm provides the Schur parameters of $u=vL$ and informs us about its quasi-definiteness: the maximum subspace $\PP_n$ where $u$ is
quasi-definite is given by the first index $n$ of inconsistency of the algorithm.

We can think in reducing the general problem to the case $r=1$ by factoring the polynomial $A$. Suppose that $A=A_1A_2$, $\deg A_1=2r_1$, $\deg
A_2=2r_2$, with $A_i$ self-reciprocal, and denote by $\mathcal{C}_j^{(1)}$, $\mathcal{C}_j^{(2)}$ the $J$-self-reciprocal matrices associated with
the direct problem $w=vA_1z^{-r_1}$, $u=wA_2z^{-r_2}$ respectively. If $\mathcal{U}_j$ are the transfer matrices for the functional $w$ with MOP
$(\xi_j)$ and $\Xi_j=\scriptsize\pmatrix{\xi_j \cr \xi_j^*}$, then $A_1\Xi_j=\mathcal{C}_j^{(1)}\Psi_{j+r_1}$ and
$A_2\Phi_j=\mathcal{C}_j^{(2)}\Xi_{j+r_2}$. This implies the equality $A\Phi_j=\mathcal{C}_j^{(2)}\mathcal{C}_{j+r_2}^{(1)}\Psi_{j+r}$, so
$\mathcal{C}_j=\mathcal{C}_j^{(2)}\mathcal{C}_{j+r_2}^{(1)}$. However, this does not always reduce a direct problem to simpler ones because the
length of the finite segments of MOP for $w$ can be not big enough to get the actual relations between all the MOP of $u$ and $v$.

\subsection{Inverse problem} \label{IP}

In this subsection we will study a problem which can be consider as the inverse of that one of the previous section. More precisely, given an
hermitian functional $u$ with a finite segment of MOP $(\varphi_j)_{j=0}^n$ and a hermitian Laurent polynomial $L$ of degree $r$, we will try to
obtain information about the hermitian solutions $v$ of $u=vL$ and their finite segments of MOP $(\psi_j)_{j=0}^m$.

First of all we will clarify the structure of the set $\{v \mbox{ hermitian} : u=vL\}$. The equation $u=vL$ is equivalent to $u[z^n] = v[z^nL]$,
$n\geq0$, which, denoting $\mu_n=u[z^n]$, $m_n=v[z^n]$ and $L(z)=\sum_{j=-r}^r \alpha_jz^j$, $\alpha_{-j}=\overline\alpha_j$, becomes
\begin{equation} \label{MOM}
\mu_n = \sum_{j=-r}^r \alpha_j m_{n+j}, \qquad n\geq 0.
\end{equation}
The first equation ($n=0$)
\begin{equation} \label{MOMENTS}
\mu_0 = 2\,\re\sum_{j=0}^r \alpha_j m_j
\end{equation}
is simply a constraint between the first $r+1$ moments $m_0,\dots,m_r$ of $v$. The rest of the equations determine the moments $m_n$, $n>r$. Since
any hermitian solution $v$ is determined by its moments $m_n$, $n\geq0$, the general solution depends on $2r$ real independent parameters obtained
establishing in the set $\{m_0,m_1,\dots,m_r\}$, $m_0\in \R$, $m_1,\dots,m_r\in\C$, the constraint (\ref{MOMENTS}).

There is another way to describe the set of hermitian solutions $v$ starting from a particular one $v_0$. Then the hermitian solutions are those
functionals with the form $v=v_0+\Delta$, where $\Delta$ is any hermitian functional satisfying $\Delta L=0$, i.e.,
$$
\Delta = \sum_{i=1}^p \sum_{k_i=0}^{q_i-1} M_{k_i}^{(i)} \delta^{(k_i}(z-\zeta_i),
\qquad M_{k_i}^{(i)}\in\C, \qquad \mbox{\rm $M_{k_j}^{(j)}=\overline{M}_{k_i}^{(i)}$
if $\zeta_j=1/\overline{\zeta}_i$},
$$
$\zeta_i$, $i=1,\dots,p$, being the roots of $A=z^{-r}L$ and $q_i$ the multiplicity of $\zeta_i$. Again we see that the hermitian solutions are
parametrized by $2r$ real parameters: the independent real and imaginary parts of the coefficients $M_{k_i}^{(i)}$. Furthermore, this approach shows
that the inverse problem is related to the study of the influence of Dirac's deltas and their derivatives on the quasi-definiteness and the MOP of a
hermitian functional.

For convenience we will denote by $H_r(u)$ the set of hermitian functionals $v$ which are solutions of $u=Lv$ for some hermitian Laurent polynomial
$L$ of degree $r$. The main result of this section characterizes the functionals of $H_r(u)$ which are quasi-definite in some subspace $\PP_m$.

\begin{thm} \label{INVERSEPROBLEM}
Let $u$ be quasi-definite in $\PP_n$.
\begin{itemize}
\item[(i)] If $n \geq r$, there is a (unique up to factors) solution $v \in H_r(u)$ quasi-definite in $\PP_{n+r}$ for each $b_1,\dots,b_{2r}\in\C\setminus\T$,
$b_{2r+1},\dots,b_{n+r}\in\C$, $\mathcal{X}_r,\dots,\mathcal{X}_n\in\J_r^\reg$ such that
\begin{equation} \label{IPMATINITIALCONDITION1}
\mathcal{X}_r \Phi_r = \Psi_{2r}, \qquad \Psi_{2r} = \mathcal{T}_{2r}\cdots\mathcal{T}_1\pmatrix{1 \cr 1},
\end{equation}
\begin{equation} \label{INVERSPROBLMSP1}
\kern-5pt \mathcal{T}_{j+r}  \mathcal{X}_{j-1} = \mathcal{X}_j \mathcal{S}_j, \qquad j=r+1,\dots,n.
\end{equation}
The relation between $v$ and $b_j$, $\mathcal{X}_j$ is that $\Psi_{j+r}=\mathcal{X}_j\Phi_j$ provides the $j+r$-th MOP of $v$ for $j=r,\dots,n$, and
$b_j\in\C\setminus\T$, $j=1,\dots,n+r$, are the first $n+r$ Schur parameters of $v$. Besides, $u=vL$ with $\det\mathcal{X}_j \propto A$,
$j=r,\dots,n$.
\item[(ii)] There is a (unique up to factors) solution $v \in H_r(u)$ quasi-definite in $\PP_{n+r}$ for each $b_1,\dots,b_r\in\C\setminus\T$,
$b_{r+1},\dots,b_{n+r}\in\C$, $\mathcal{X}_0,\dots,\mathcal{X}_n\in\J_r^\reg$ such that
\begin{equation} \label{IPMATINITIALCONDITION2}
\mathcal{X}_0 \pmatrix{1 \cr 1} = \Psi_r, \qquad \Psi_r = \mathcal{T}_r\cdots\mathcal{T}_1\pmatrix{1 \cr 1},
\end{equation}
\begin{equation} \label{INVERSPROBLMSP2}
\kern-22pt \mathcal{T}_{j+r}  \mathcal{X}_{j-1} = \mathcal{X}_j \mathcal{S}_j, \qquad j=1,\dots,n.
\end{equation}
The relation between $v$ and $b_j$, $\mathcal{X}_j$ is that $\Psi_{j+r}=\mathcal{X}_j\Phi_j$ provides the $j+r$-th MOP of $v$ for $j=0,\dots,n$, and
$b_j\in\C\setminus\T$, $j=1,\dots,n+r$, are the first $n+r$ Schur parameters of $v$. Besides, $u=vL$ with $\det\mathcal{X}_j \propto A$,
$j=0,\dots,n$.

\end{itemize}
\end{thm}

\begin{proof}
We will prove only (i), the proof of (ii) is similar. Bearing in mind Theorems \ref{EQGENERALPROBLEMTHMF} and \ref{EQUIVALENCESRRGC} we only need to
show that (\ref{IPMATINITIALCONDITION1}) and (\ref{INVERSPROBLMSP1}) imply that $\mathcal{X}_j\Phi_j$ gives for $j=r,\dots,n$ the $j+r$-th MOP of a
unique $v \in H_r(u)$ whose first $n+r$ Schur parameters are $b_j$, $j=1,\dots,n+r$.

Let us define $\Psi_j=\mathcal{T}_j\cdots\mathcal{T}_1\scriptsize\pmatrix{1 \cr 1}$. Since $\mathcal{X}_j\in\J_r^\reg$, $j=r,\dots,n$, recurrence
(\ref{INVERSPROBLMSP1}) implies that $|b_j|\neq1$, not only for $j=1,\dots,2r$, but also for $j=2r+1,\dots,n+r$. Therefore, $(\psi_j)_{j=0}^{n+r}$ is
a finite segment of MOP with respect to some hermitian functional $\hat v$.

From (\ref{IPMATINITIALCONDITION1}), (\ref{INVERSPROBLMSP1}) and the recurrence relation for $(\Phi_j)_{j=0}^n$ we obtain for $j=r,\dots,n$,
$$
\Psi_{j+r} = \mathcal{T}_{j+r}\cdots \mathcal{T}_{2r+1} \Psi_{2r} = \mathcal{T}_{j+r} \cdots \mathcal{T}_{2r+1} \mathcal{X}_r \Phi_r =
\mathcal{X}_j\mathcal{S}_j\cdots\mathcal{S}_{r+1}\Phi_r = \mathcal{X}_j \Phi_j.
$$
Hence, Theorem \ref{EQGENERALPROBLEMTHMF} proves that $u \equiv \hat v \hat L$ in $\PP_n$ for some hermitian Laurent polynomial $\hat L$ of degree
$r$. Multiplying $\hat L$ by a real factor we can get a hermitian Laurent polynomial $L$ of degree $r$ such that $u=\hat vL$ in $\PP_n$.

The equality $u=\hat vL$ in $\PP_n$, as well as the fact that $(\psi_j)_{j=0}^{n+r}$ is a finite segment of MOP for $\hat v$, only depends on the
first $n+r+1$ moments $\hat v[z^j]$, $j=0,\dots,n+r$, of $\hat v$. Let us define a new hermitian functional $v$ fixing its moments $m_j=v[z^j]$ by
$m_j=\hat v[z^j]$ for $j \leq n+r$, and $m_j$ given by (\ref{MOM}) for $j \geq n+r+1$. Then $v$ is a solution of $u=vL$, has $(\psi_j)_{j=0}^{n+r}$
as a finite segment of MOP and its first $n+r$ Schur parameters are $\psi_j(0)=b_j$, $j=1,\dots,n+r$.

Finally, the first $n+r$ Schur parameters of a functional $v$ determine its finite segment of MOP of length $n+r+1$ and, thus, its first $n+r+1$
moments up to a common factor. Requiring also $u=vL$ for a given hermitian Laurent polynomial of degree $r$ fixes the rest of the moments up to the
common factor due to (\ref{MOM}). Therefore, the conditions of (i) define a unique hermitian functional $v$ up to factors because $L$ is determined
up to real factors by $\det\mathcal{X}_j$.
\end{proof}

We have the following relation between the Schur parameters of $u$ and $v \in H_r(u)$. To prove it simply choose $z=0$ in the equivalent version
$\mathcal{B}_{j+r}  \tilde{\mathcal{X}}_{j-1} = \mathcal{X}_j \mathcal{A}_j$ of (\ref{INVERSPROBLMSP2}) and use that $X_j$ is monic, i.e.,
$X_j^*(0)=1$.

\begin{cor} \label{SCHURPARAMETERSIP}
The $j+r$-th Schur parameter $b_{j+r}$ of $v \in H_r(u)$ can be obtained from the $j$-th Schur parameter $a_j$ of $u$  by
\begin{equation}\label{SCHPINVPROBL}
b_{j+r} = { a_{j} -  \overline{Y_{j-1}^*(0)} \over  \overline{X_{j-1}(0)}}.
\end{equation}
\end{cor}

Theorem \ref{INVERSEPROBLEM} and Corollary \ref{SCHURPARAMETERSIP} provide algorithms generating the solutions of the inverse problem which are
quasi-definite in some subspace $\PP_m$. The algorithms are based on the consistence of recurrence (\ref{INVERSPROBLMSP1}) or
(\ref{INVERSPROBLMSP2}), what can be defined in a similar way to the case of Algorithm D. We have several possibilities depending of the initial
data.

If we know that $L$ has degree $r$ but not its explicit form, we can proceed in the following ways, depending whether we are interested in the
solutions which are quasi-definite (at least) in $\PP_{2r}$ or $\PP_r$.

\bigskip

\noindent{\bf Algorithm I1}
\begin{itemize}
\item Choice of $\Psi_{2r}$, i.e., of $b_1,\dots,b_{2r}\in \C\setminus\T$.
\item Determination of $\mathcal{X}_r\in\J_r$ from initial condition (\ref{IPMATINITIALCONDITION1}) and $\Phi_r$, $\Psi_{2r}$.
\item For $j=r+1,r+2,\dots$
    \begin{itemize}
    \item[$\bullet$] While $X_{j-1}(0)\neq 0$, calculation of $b_{j+r}$ from (\ref{SCHPINVPROBL}) and $a_j$, $\mathcal{X}_{j-1}$.
    \item[$\bullet$] Determination of $\mathcal{X}_j\in\J_r$ from recurrence (\ref{INVERSPROBLMSP1}) and $a_j$, $b_{j+r}$, $\mathcal{X}_{j-1}$.
    \end{itemize}
\end{itemize}

\smallskip

\noindent{\bf Algorithm I2}
\begin{itemize}
\item Choice of $\Psi_r$, i.e., of $b_1,\dots,b_r\in \C\setminus\T$.
\item Choice of a solution $\cX_0\in\J_r^\reg$ of initial condition (\ref{IPMATINITIALCONDITION2}) using $\Psi_r$, i.e.,
      choice of a monic polynomial $X_0$ of degree $r$ with $X_0(0)\neq0$ and determination of $Y_0=\psi_r-X_0$.
\item For $j=1,2,\dots$
    \begin{itemize}
    \item[$\bullet$] While $X_{j-1}(0)\neq 0$, calculation of  $b_{j+r}$ from (\ref{SCHPINVPROBL}) and $a_j$, $\mathcal{X}_{j-1}$.
    \item[$\bullet$] Determination of $\mathcal{X}_j\in\J_r$ from recurrence (\ref{INVERSPROBLMSP2}) and $a_j$, $b_{j+r}$, $\mathcal{X}_{j-1}$.
    \end{itemize}
\end{itemize}

For any of these two algorithms we recover the polynomial perturbation through $A\propto\det\mathcal{X}_j$.

On the contrary, if we know explicitly the hermitian polynomial $L$ of degree $r$, we have the following scheme to find the solutions $v$ of $u=vL$
which are quasi-definite in $\PP_r$.

\bigskip

\noindent{\bf Algorithm I3}

\begin{itemize}
\item Choice of $\Psi_{r}$ i.e., of $b_1,\dots,b_r\in \C\setminus\T$.
\item Determination of $\displaystyle\mathcal{X}_0 = {\adj(\mathcal{C}_0)\over C_0(0)}$  from initial condition (\ref{DMATINITIALCONDITION})
      and $\Psi_r$, $A$.
\item For $j=1,2,\dots,$
    \begin{itemize}
    \item[$\bullet$] While $X_{j-1}(0)\neq 0$, calculation of  $b_{j+r}$ from (\ref{SCHPINVPROBL}) and $a_j$, $\mathcal{X}_{j-1}$.
    \item[$\bullet$] Determination of $\mathcal{X}_j\in\J_r$ from recurrence (\ref{INVERSPROBLMSP2}) and $a_j$, $b_{j+r}$, $\mathcal{X}_{j-1}$.
    \end{itemize}
\end{itemize}

We can assure that any step of the above algorithms generates a matrix $\mathcal{X}_j\in\J_r$ due to Lemma \ref{CT=SC} (ii) and the fact that
$X_{j-1}(0)\neq0$ is equivalent to $\mathcal{X}_{j-1}\in\J_r^\reg$ when $\mathcal{X}_{j-1}\in\J_r$.

The $n$-consistence of the above algorithms, which means that they work for $j\leq n$, is equivalent to the existence of a finite segment of MOP of
length $n+r$ for the corresponding solution $v$ of $u=vL$. Such $n$-consistence can be written as $X_j(0)\neq 0$, $j\leq n-1$, which holds iff
$|b_j|\neq1$, $j\leq n+r-1$.

Comparing the above algorithms we see that the arbitrariness in the parameters $b_{r+1},\dots,b_{2r}$ is equivalent to the arbitrariness of the
polynomial modification $L$ of degree $r$. This means that any of the infinitely many solutions $\mathcal{X}_0\in\J_r^\reg$ of
$\mathcal{X}_0\Phi_0=\Psi_r$ should be determined by $\det\mathcal{X}_0$, a result which is proved in the next proposition.

\begin{prop} \label{REL-IC}

Given $b_1,\dots,b_r\in\C\setminus\T$ and a self-reciprocal polynomial $A$ of degree $2r$, there exist a unique solution $\cX_0\in\J_r^\reg$ of
$\cX_0\pmatrix{1 \cr 1}=\Psi_r$, $\Psi_r=\mathcal{T}_r\cdots\mathcal{T}_1\pmatrix{1 \cr 1}$, such that $\det\cX_0 \propto A$.

\end{prop}

\begin{proof}
Given $\Psi_r$, each solution of $\cX_0\Phi_0=\Psi_r$ with the form
$$
\cX_0 = \pmatrix{X_0 & Y_0 \cr zY_0^* & X_0^*}, \quad \deg X_0 = r, \quad \deg Y_0 \leq r-1,
$$
is determined by a monic polynomial $X_0$ because $Y_0=\psi_r-X_0$. Then $\mathcal{X}_0\in\J_r^\reg$ iff $X_0(0)\neq0$. Therefore,
$\det\cX_0=X_0^*\psi_r+X_0\psi_r^*-\psi_r\psi_r^*$. Hence, if $A$ is a self-reciprocal polynomial of degree $2r$ and $\lambda\in\R$,
\begin{equation} \label{DET}
\det\cX_0 = \lambda A \kern7pt \Leftrightarrow \kern7pt \lambda A + \psi_r\psi_r^* = X_0^*\psi_r + X_0\psi_r^*.
\end{equation}

From Remark \ref{BASISNR} we know that $\lambda A + \psi_r\psi_r^* = C\psi_r+D\psi_r^*$ for some polynomials $C\in\PP_r$, $D\in\PP_{r-1}$. Since
$\lambda A + \psi_r\psi_r^*$ is self-reciprocal in $\PP_{2r}$, $(C-zD^*)\psi_r=(C^*-D)\psi_r^*$, so $C^*-D=c\psi_r$ and $C-zD^*=c\psi_r^*$ for some
$c\in\R$. Then, the identity
$$
\lambda A + \psi_r\psi_r^* = (C-\frac{c}{2}\psi_r^*)\psi_r + (D+\frac{c}{2}\psi_r)\psi_r^* = \frac{1}{2}(C+zD^*)\psi_r + \frac{1}{2}(C^*+D)\psi_r^*
$$
proves that (\ref{DET}) holds with $X_0=\frac{1}{2}(C^*+D)$, thus $\cX_0$ satisfies $\det\cX_0 = \lambda A$ with such a
choice. Furthermore, $X_0$ is monic of degree $r$ iff $X_0^*(0)=1$, which (\ref{DET}) shows that corresponds to $\lambda=X_0(0)/A(0)$.

Now, let $\cX_0,\hat\cX_0\in\J_r^\reg$ be such that $\cX_0\Phi_0=\hat\cX_0\Phi_0=\Psi_r$. Assume that $\det\hat\cX_0=\lambda\det\cX_0$ for some
$\lambda\in\R$. Using an obvious notation, this means that $\hat{X}_0^*\psi_r + (\hat{X}_0-\psi_r)\psi_r^* = \lambda(X_0^*\psi_r +
(X_0-\psi_r)\psi_r^*)$. The uniqueness of the polynomials $C,D$ in Remark \ref{BASISNR} ensures that $\hat{X}_0^*=\lambda X_0^*$ and
$\hat{X}_0-\psi_r=\lambda(X_0-\psi_r)$, which implies that $\hat\cX_0=\cX_0$. \epr

The previous results show that the solutions of the inverse problem are parametrized by their first $r$ or $2r$ Schur parameters, depending on
whether we fix the polynomial perturbation or only its degree. Of course, such a parametrization works only for the solutions which are
quasi-definite (at least) in $\PP_r$ and $\PP_{2r}$ respectively. Each of these solutions will have a finite segment of MOP of maximum length
determined by the consistence level of the corresponding algorithm.

\subsubsection{The case $r=1$}

As an example of the previous discussion we will analize the particular case of the inverse problem corresponding to a hermitian Laurent polynomial
perturbation $L$ of degree 1. So, we consider the MOP $(\varphi_j)$ with respect to a hermitian linear functional $u$ and we define the monic
polynomials $(\psi_j)$
\begin{equation} \label{CP}
\psi_{j+1} = (z + x_j ) \varphi_j  + y_j \varphi_j^*, \qquad j \geq 0,
\end{equation}
with $\psi_0(z) = 1$ and  $x_j,$ $y_j\in \mathbb{C}$. The polynomials $(\psi_j)$ are the only candidates to be MOP of a solution $v$ of $u=vL$.

We can write (\ref{CP}) in a matrix form as
\begin{equation} \label{CPMF}
\Psi_{j+1} = \mathcal{X}_j \Phi_j,\qquad \mathcal{X}_j = \pmatrix{z+x_j&y_j \cr \overline{y}_j z & 1 + \overline{x}_j z}, \qquad j\geq 0,
\end{equation}
and (\ref{INVERSPROBLMSP2}) becomes
\begin{equation} \label{MATIPS}
\cases{
x_{j-1} + \overline{y}_{j-1} b_{j+1} = x_j + y_j \overline{a}_j,
\cr
b_{j+1} = x_j a_j + y_j,
\cr
\overline{x}_{j-1} b_{j+1} + y_{j-1} = a_j,
}
\end{equation}
or equivalently,
\begin{equation} \label{MATIPS2}
b_{j+1} = \frac{a_j-y_{j-1}}{\overline x_{j-1}},
\qquad
\pmatrix{1&\overline{a}_j \cr a_j&1} \pmatrix{x_j \cr y_j} = \pmatrix{x_{j-1}+\overline{y}_{j-1}b_{j+1} \cr b_{j+1}},
\end{equation}
So, Algorithm I2 reads as follows:
\begin{itemize}
\item Choice of $b_1\in \C\setminus\T$ and $x_0\in\C^*$ which determines $y_0=b_1-x_0$.
\item For $j=1,2,\dots,$ while $x_{j-1}\neq0$, calculation of $b_{j+1}$, $x_j$, $y_j$ from $a_j$, $x_{j-1}$, $y_{j-1}$ using (\ref{MATIPS2}).
\end{itemize}

For any choice of $x_0$ we can recover the polynomial perturbation through $A \propto \det\mathcal{X}_0 = \overline{x}_0 z^2 + (1 + |x_0|^2 -
|y_0|^2)z + x_0$. According to Proposition \ref{REL-IC}, given $b_1$, each choice of $x_0$ in the previous algorithm provides a solution of the
inverse problem corresponding to a different polynomial perturbation. These solutions have well defined MOP $\psi_0$, $\psi_1$, so the algorithm
provides all the solutions of the inverse problem which are quasi-definite at least in $\PP_1$. The maximum length of the finite segments of MOP for
a particular solution is equal to the consistence level of the algorithm starting with the values $b_1$ and $x_0$ defining such solution.

It is remarkable that, when $r = 1$, the consistence of Algorithm I2 is equivalent to the compatibility of (\ref{INVERSPROBLMSP2}), i.e., any
solution of (\ref{MATIPS}) for $j \leq n$ starting with $x_0\neq0$ necessarily satisfies $x_j\neq0$ for $j \leq n-1$. We can see this by induction:
if (\ref{MATIPS}) has a solution for $j \leq n+1$, then $x_{n-1}\neq0$ due to the induction hypothesis, so $x_n=0$ would give $b_{n+1}\in\T$
according to (\ref{INVERSPROBLMSP2}); on the other hand, setting $x_n=0$ in (\ref{MATIPS}) for $j=n,n+1$ we get $y_n=b_{n+1}$ and $y_n=a_{n+1}$,
which is a contradiction because $a_{n+1}\notin\T$.

\subsubsection{An example of the inverse problem} \label{EX-INV}

As an application, we will solve the inverse problem for an arbitrary hermitian polynomial perturbation $L$ of degree 1, when $u$ is the functional
associated with the Lebesgue measure on the unit circle
$$
dm(z)=\frac{1}{2\pi i}\frac{dz}{z}=\frac{d\theta}{2\pi}, \qquad z=e^{i\theta}.
$$
More precisely, we will characterize the quasi-definite solutions $v \in H_1(u)$. Indeed, we will do something more than this because our methods
permits us to characterize all the solutions $v \in H_1(u)$ which are quasi-definite at least in $\PP_1$, providing also the maximum subspace $\PP_m$
where each of such solutions is quasi-definite.

Bearing in mind the comments at the beginning of Section \ref{IP}, and taking into account the possibilities for the roots of a self-reciprocal
polynomial $A$ of degree 2, this is equivalent to the analysis of functionals $v$ with the form
\begin{itemize}
\item[(a)]
$\ds v_0 + M\delta(z-\zeta) + \overline{M}\delta(z-1/\overline\zeta), \quad \zeta\in\D, \quad M\in\C$,
\item[(b)]
$\ds v_0 + M_1\delta(z-\zeta) + M_2\delta'(z-\zeta), \quad \zeta\in\T, \quad M_i\in\R$,
\item[(c)]
$\ds v_0 + M_1\delta(z-\zeta_1) + M_2\delta(z-\zeta_2), \quad \zeta_1\neq\zeta_2, \quad \zeta_i\in\T, \quad M_i\in\R,$
\end{itemize}
where $v_0$ is a particular solution of the inverse problem. In case (a) we can take $v_0$ as a multiple of the functional associated with the
measure $dm(z)/|z-\zeta|^2$ and then (a) is known as the Geronimus transformation of the Lebesgue measure. The Geronimus transformation of an
arbitrary positive measure on the unit circle has been studied in \cite{GaHeMa09,Ga09}, while a more general laurent polynomial transformation has
been analyzed in \cite{Su93,Ca,GaMa1,GaMa2,CaGaMa}. Our approach permits us to deal with the above three transformations simultaneously.

The functional $u$ is positive definite with MOP $\varphi_n(z) = z^n$, $n\geq0$, and Schur parameters $a_n=0$, $n\geq1$, so that (\ref{MATIPS})
becomes
\begin{equation} \label{INV-LEB}
\cases{
x_{n-1} + b_{n+1} \overline{y}_{n-1} = x_n,
\cr
b_{n+1} =  y_n,
\cr
x_{n-1} \overline{b}_{n+1} + \overline{y}_{n-1} = 0.
}
\end{equation}
Following Algorithm I2, every choice of $b_1\in\C\setminus\T$ and $x_0\in\C^*$ determines $y_0=b_1-x_0$ providing initial conditions for the above
recurrence. Each of such initial conditions is associated with a different solution of the inverse problem we are considering, and this solution is
quasi-definite exactly when the related initial conditions make (\ref{INV-LEB}) compatible for every $n\in\N$, i.e., $x_n\neq0$ for all $n$. The
corresponding orthogonal polynomials $(\psi_n)$ are
$$
\psi_{n+1}(z) = (z+x_n)z^n+y_n.
$$

The second equation in (\ref{INV-LEB}) permits us to eliminate $b_n$ and formulate equivalently the recurrence only in terms of $x_n$ and $y_n$,
\begin{equation} \label{SEJ1}
\cases{
\ds x_n = {|x_{n-1}|^2 - |y_{n-1}|^2 \over \overline{x}_{n-1}},
\medskip \cr
\ds y_n= b_{n+1} = -{y_{n-1} \over\overline{x}_{n-1}}.
}
\end{equation}
The second equation in (\ref{SEJ1}) is solved by
\begin{equation} \label{y}
y_n = (-1)^n\frac{y_0}{\overline{x}_0\cdots\overline{x}_{n-1}},
\end{equation}
so we only must care about the first equation in (\ref{SEJ1}).

If $L=P+P_*$ with $P(z)= \alpha z + \beta$, $\alpha\in\C^*$, $\beta\in\R$, we know that
$$
\det\cX_n(z) = \overline{x}_n z^2 + (1+ |x_n|^2 - |y_n|^2) z + x_n   \propto A(z) = \alpha z^2 + 2\beta z + \overline{\alpha}.
$$
Therefore,
\begin{equation} \label{DETX-A-LEB}
{x_n \over \overline{x}_n} = {\overline{\alpha} \over \alpha},
\qquad
{1 + |x_n|^2 - |y_n|^2 \over \overline{x}_n} = 2{\beta \over \alpha}.
\end{equation}
This implies that $x_n = s_n {\overline{\alpha}\over |\alpha|}$, $s_n \in \mathbb{R}$, and the first equation of (\ref{SEJ1}) is equivalent to
\begin{equation} \label{INV-LEB2}
x_n = 2\tilde{\omega} - {1 \over \overline{x}_{n-1}}, \qquad \tilde{\omega} = {\beta \over \alpha}.
\end{equation}
That is, we have reduced the compatibility of (\ref{INV-LEB}) to the compatibility of (\ref{INV-LEB2}) for $x_n$, which can be rewritten in terms of
$s_n$ as
\begin{equation} \label{INV-LEB3}
s_n = 2\omega - {1\over s_{n-1}}, \qquad  \omega = {\beta \over |\alpha|},
\end{equation}
while the compatibility means simply that $s_n\neq0$ for all $n$. If $s_j\neq0$ for $j<n$ but $s_n=0$ then the related solution is not quasi-definite
but has only the first $n+1$ MOP $\psi_0,\dots,\psi_n$.

The key idea to calculate $s_n$ is to write (\ref{INV-LEB3}) as a continued fraction
$$
s_n = 2\omega - {\;\;1\;|\over{|2\omega}} - {\;\;1\;|\over{|2\omega}} - \cdots - {\;\;1\;|\over{|2\omega}} - {\;\;1\;|\over{|\,s_0}}.
$$
According to the general theory of continued fractions (see for instance \cite{Wa48}),
$$
s_n = {s_0 Q_{n-1} - Q_{n-2} \over s_0 P_{n-1} - P_{n-2}},
$$
where $P_n$ and $Q_n$ satisfy the difference equations
$$
\begin{array}{l}
Q_k = 2\omega Q_{k-1} - Q_{k-2}, \qquad Q_0=2\omega, \quad Q_{-1}=1,
\medskip \cr
P_k = 2\omega P_{k-1} - P_{k-2}, \qquad P_0=1, \quad P_{-1}=0.
\end{array}
$$
Since $P_1=2\omega=Q_0$, we get $Q_k=P_{k+1}$.

On the other hand, the recurrence and initial conditions for $P_k$ show that $P_k=U_k(\omega)$, where $U_k$ is the second kind Tchebyshev polynomial
of degree $k$,
$$
U_k(\omega)={\lambda^{k+1}-\lambda^{-(k+1)} \over \lambda-\lambda^{-1}}, \qquad \lambda = \omega + \sqrt{\omega^2-1}.
$$
The parameter $\lambda$ is one of the roots of the characteristic polynomial
\begin{equation} \label{CHAR}
B(z) = A(-z\overline{\alpha}/|\alpha|) = z^2 - 2\omega z + 1,
\end{equation}
no matter which one because both of them are inverse of each other.

Hence,
$$
s_n = {s_0U_n(\omega)-U_{n-1}(\omega) \over s_0U_{n-1}(\omega)-U_{n-2}(\omega)}, \qquad n\geq1.
$$
As a consequence, the solution of the inverse problem is quasi-definite if and only if
\begin{equation} \label{QD-INV-LEB}
s_0 \, U_n(\omega) \neq U_{n-1}(\omega), \qquad n\geq0.
\end{equation}
In case $s_0 \, U_j(\omega) \neq U_{j-1}(\omega)$ for $j<n$ but $s_0 \, U_n(\omega) = U_{n-1}(\omega)$, the related solution of the inverse problem
is quasi-definite in $\PP_n$ but not in $\PP_{n+1}$.

Besides, from the solution for $s_n$ we can obtain the rest of the variables of interest for the inverse problem. In particular, for $n\geq1$,
$$
\begin{array}{l}
\ds x_n = \frac{\overline\alpha}{|\alpha|} {s_0U_n(\omega)-U_{n-1}(\omega) \over s_0U_{n-1}(\omega)-U_{n-2}(\omega)},
\medskip \cr
\ds b_{n+1} = y_n = \left(-\frac{\overline\alpha}{|\alpha|}\right)^n \frac{y_0}{s_0U_{n-1}(\omega)-U_{n-2}(\omega)}.
\end{array}
$$

We can express these variables, as well as the quasi-definiteness condition (\ref{QD-INV-LEB}), in terms of other parameters. For instance, following
Algorithm I2, we can use as free parameters $b_1$ and $x_0$. Then, using (\ref{DETX-A-LEB}) and the relation $y_0=b_1-x_0$, we get for some
$\kappa\in\R^*$,
\begin{equation} \label{ab-b1x0}
\alpha = \kappa \overline{x}_0, \qquad \beta = \frac{\kappa}{2} \left(1-|b_1|^2+2\,\re(\overline{x}_0b_1)\right).
\end{equation}

If we chose the approach of Algorithm I3, then the free parameters must be $b_1$ and $\alpha,\beta$, so we should express $s_0$, $x_0$ and $y_0$ in
terms of them. From (\ref{ab-b1x0}), bearing in mind that $\kappa=|\alpha|/s_0$, we obtain
\begin{equation} \label{r0x0y0} \kern-5pt
s_0 = \frac{|\alpha|}{2} \frac{1-|b_1|^2}{\beta-\re(\alpha b_1)}, \kern9pt x_0 = \frac{\overline\alpha}{2} \frac{1-|b_1|^2}{\beta-\re(\alpha b_1)},
\kern9pt y_0 = -\frac{1}{2} \frac{A(-b_1)}{\beta-\re(\alpha b_1)}.
\end{equation}

Finally, we can use the point of view of Algorithm I1. This implies that we restrict our attention to the solutions of the inverse problem which are
quasi-definite at least in $\PP_2$, an not only in $\PP_1$, which was the case till now. Then, according to Algorithm I1, $b_1$ and $b_2$ could be
used as free parameters too. This can be done using (\ref{ab-b1x0}) and the relation
$$
b_2 = y_1 = - \frac{y_0}{\overline{x}_0} = \frac{x_0-b_1}{\overline{x}_0},
$$
which determines $x_0$ as the following function of $b_1$ and $b_2$,
$$
x_0 = \frac{1}{1-|b_2|^2}(b_1+\overline{b}_1b_2).
$$
Also, $b_2$ can be expressed in terms of $\alpha$, $\beta$ and $b_1$ using (\ref{r0x0y0}), which gives
$$
b_2 = \frac{A(-b_1)}{\alpha(1-|b_1|^2)}.
$$

The fact that the iterations (\ref{INV-LEB3}) generating the solutions of the inverse problem and the quasi-definiteness condition (\ref{QD-INV-LEB})
are given in terms of $\alpha$, $\beta$ and $s_0$ uniquely suggests the possibility of using these variables to parameterize such solutions. However,
this is not possible because an arbitrary value of $\alpha$, $\beta$ and $s_0$ can be associated with no value of $b_1$ or with infinitely many
values of $b_1$. Indeed, the first identity of (\ref{r0x0y0}) can be written as
$$
|b_1-x_0|^2 = B(s_0),
$$
which shows that we have the following possibilities:
\begin{itemize}
\item If $B(s_0)<0$ there is no solution associated with $\alpha$, $\beta$ and $s_0$.
\item If $B(s_0)=0$ there is exactly one solution associated with $\alpha$, $\beta$ and $s_0$: that one determined by $\alpha$, $\beta$ and
      $b_1=x_0=s_0\overline\alpha/|\alpha|$.
\item If $B(s_0)>0$ there are infinitely many solutions associated with $\alpha$, $\beta$ and $s_0$: those ones determined by $\alpha$, $\beta$ and
      any value of $b_1$ in the circle with center $x_0$ and radius $\sqrt{B(s_0)}$. Therefore such solutions are parametrized by a phase.
\end{itemize}
In consequence, given $P(z) = \alpha z + \beta$, the inequality $B(s_0)\geq0$ determines the permitted values of $s_0$. The set of solutions
associated with $P$ and a permitted value $s_0$ will be called the circle of solutions for $P$ and $s_0$, and will be denoted $C(P,s_0)$. Eventually
$B(s_0)=0$ and $C(P,s_0)$ degenerates into a single solution. The fact that the quasi-definiteness condition depends only on $\omega=\beta/|\alpha|$
and $s_0$ means that all the solutions of $C(P,s_0)$ have the same number of MOP.

It seems that the presence of the circles of solutions with similar properties should have to do with some symmetry of the problem. The most obvious
one is the rotation symmetry. If $u=vL$, then $u_\theta=v_\theta L_\theta$ for any angle $\theta$, where the rotation of a Laurent polynomial $f$ and
a functional $v$ are defined by $f_\theta(z)=f(e^{-i\theta}z)$ and $v_\theta[f]=v[f_{-\theta}]$. When $u_\theta=u$ we find that $v \in H_1(u)$
implies $v_\theta \in H_1(u)$. The only functional $u$ which is invariant under any rotation is that one defined by the Lebesgue measure, so only in
this case we can assure that $H_1(u)$ is constituted by ``circles of solutions" obtained by the rotation of one of them.

Bearing in mind that we are identifying equivalent functionals and that the rotation of a functional preserves its quasi-definiteness properties, the
rotation symmetry permits us to reduce the analysis of the set $H_1(u)$ for the Lebesgue functional $u$ to the case $\alpha=1$ because each ``circle
of solutions" has a representative with a monic polynomial $P$. However, the reduction of the analysis to such canonical cases is not possible for
any other hermitian functional $u$.

Nevertheless, the rotation symmetry of the Lebesgue measure is not responsible of the circles of solutions $C(P,s_0)$ that we have found: the
solutions of any circle $C(P,s_0)$ have a common polynomial $P$, while the solutions of a ``circle of solutions" associated with the rotation
symmetry are related to different polynomials $P$ obtained by a rotation of one of them; furthermore, the rotation of a functional also rotates its
Schur parameters around the origin, but the parameters $b_1$ of the solutions of a circle $C(P,s_0)$ are obtained rotating one of them around
$x_0\neq0$. The search for the ``symmetry transformations" relating the functionals of a circle $C(P,s_0)$ remains as an open problem.

Some particular quasi-definite solutions deserve a special mention, i.e., the solutions with constant coefficients $x_n$, $y_n$, which are
characterized by any of the statements of the following equivalence, which follows easily from the previous results,
$$
\begin{array}{c}
s_n=s_0,  n\geq 0
 \Leftrightarrow
x_n=x_0,  n\geq 0
 \Leftrightarrow
y_n=0,  n\geq0
 \Leftrightarrow
b_n=0,  n\geq 2
 \Leftrightarrow
\medskip \cr
\Leftrightarrow
b_2=0
 \Leftrightarrow
y_0=0
 \Leftrightarrow
b_1=x_0
 \Leftrightarrow
A(-b_1)=0
 \Leftrightarrow
A(-x_0)=0
 \Leftrightarrow
B(s_0)=0.
\end{array}
$$
Therefore, these constant solutions correspond exactly to the case where a circle of solutions degenerates into a single solution. The corresponding
functionals are those ones associated with the Bernstein-Szeg\H{o} polynomials $\psi_{n+1}(z)=(z+b_1)z^n$. Since $-b_1$ must be a root of $A$, such
solutions can appear only when $A$ has roots outside the unit circle, which corresponds to the Geronimus transformation of the Lebesgue measure.

It is advisable to discuss the three possibilities (a), (b), (c) pointed out at the beginning of Section \ref{EX-INV} according to the location of
the roots of the polynomial $A$. The reason is that the qualitative behaviour of the solutions of the inverse problem depend strongly on the case at
hand. Before doing this we must remark that, since $B(z)=A(-\lambda\overline\alpha/|\alpha|)$, the roots $\zeta_1,\zeta_2$ of $A$ are related to the
roots $\lambda,\lambda^{-1}$ of $B$ through $\zeta_1=-\lambda\overline\alpha/|\alpha|$, $\zeta_2=-\lambda^{-1}\overline\alpha/|\alpha|$, and the
three cases we want to discuss can be characterized in terms of $\omega$. Concerning this discussion, notice that, once $b_1$ is fixed, any
restriction on $\omega$ becomes a restriction on the initial value $x_0$ by (\ref{ab-b1x0}).

We will comment the asymptotics in each of the cases (a), (b), (c) using the notation $p_n \sim q_n$ to mean that $\lim(p_n/q_n)=1$.

\begin{itemize}

\item[(a)] $A(z) = \alpha(z-\zeta)(z-1/\overline\zeta), \quad \zeta\in\D \quad \Leftrightarrow \quad |\omega|>1$.

This case corresponds to $B$ having two different roots $\lambda,\lambda^{-1}\in\R$, thus we can suppose $|\lambda|<1$ so that
$\zeta=-\lambda\overline\alpha/|\alpha|$. Then, the quasi-definiteness condition (\ref{QD-INV-LEB}) becomes \beq \label{QDa} s_0 \neq \lambda
\frac{1-\lambda^{2n}}{1-\lambda^{2n+2}}, \qquad n\geq0, \eeq or equivalently
$$
x_0 \neq -\zeta \frac{1-|\zeta|^{2n}}{1-|\zeta|^{2n+2}}, \qquad n\geq0,
$$
which can be also understood as a restriction on $b_1$ because, together with $A$, it determines $x_0$ through (\ref{r0x0y0}).
\newline
\hspace*{10pt} Given only $\alpha$ and $\beta$, not any value of $s_0$ is permitted because $B(s_0)$ can be negative. This happens when
$\lambda_1<s_0<\lambda_2$, where $\lambda_1,\lambda_2$ are the roots $\lambda,\lambda^{-1}$ of $B$ but ordered so that $\lambda_1<\lambda_2$.
Therefore, the values of $s_0$ associated with a solution of the inverse problem are those lying on $(-\infty,\lambda_1]\cup[\lambda_2,\infty)$.
Then, the corresponding sequence of MOP is infinite or finite depending on whether the quasi-definiteness condition (\ref{QDa}) is satisfied for any
$n$ or not.
\newline
\hspace*{10pt} There are two quasi-definite constant solutions: $s_n=\lambda$, $x_n=-\zeta=b_1$, $y_n=0$ and $s_n=\lambda^{-1}$,
$x_n=-1/\overline\zeta=b_1$, $y_n=0$. Both of them give rise to a Bernstein-Szeg\H{o} solution with $b_n=0$, $n\geq2$, but the first one is positive
definite with measure $dm(z)/|z-\zeta|^2$, while the second one is indefinite. As we will see, the solution $dm(z)/|z-\zeta|^2$ is somewhat singular
among the solutions of the inverse problem, so in what follows we will consider only $s_0\neq\lambda$, i.e., $x_0\neq-\zeta$. Then,
$$
\begin{array}{l}
\ds s_0U_n(\omega)-U_{n-1}(\omega) \sim \frac{s_0-\lambda}{1-\lambda^2}\,\lambda^{-n}, \qquad b_2 =
\frac{(b_1+\zeta)(b_1+1/\overline\zeta)}{1-|b_1|^2},
\medskip \cr
\ds b_{n+1} = y_n \sim b_2\frac{x_0(1-|\zeta|^2)}{x_0+\zeta}\,\zeta^{n-1} =
-\overline\alpha\frac{b_1+1/\overline\zeta}{\overline\alpha\overline{b}_1+\alpha\zeta}(1-|\zeta|^2)\zeta^{n-1},
\medskip \cr
\ds \lim b_n = \lim y_n = 0, \qquad \lim s_n = \lambda^{-1}, \qquad \lim x_n = -1/\overline\zeta.
\end{array}
$$
Furthermore, the related orthogonal polynomials obey the asymptotics
$$
\ba{l}
\ds \psi_{n+1}(z) \sim -\overline\alpha\frac{b_1+1/\overline\zeta}{\overline\alpha\overline{b}_1+\alpha\zeta}(1-|\zeta|^2)\zeta^{n-1},
\qquad |z|<|\zeta|,
\medskip \cr
\ds \psi_{n+1}(z) \sim \left(z-1/\overline\zeta\right)z^n, \kern89pt |z|>|\zeta|.
\ea
$$
\hspace*{10pt} We observe that the parameters of the indefinite Bernstein-Szeg\H{o} solution provide the asymptotics of the parameters for all the
solutions except for $dm(z)/|z-\zeta|^2$. Also, the indefinite Bernstein-Szeg\H{o} polynomials $(z-1/\overline\zeta)z^n$ yield the large $z$
asymptotics of the rest of MOP which solve the inverse problem, with the exception again of the positive definite Bernstein-Szeg\H{o} ones
$(z-\zeta)z^n$.

\item[(b)] $A(z) = \alpha(z-\zeta)^2, \quad \zeta\in\T \quad \Leftrightarrow \quad |\omega|=1$.

This is equivalent to state that $B$ has a double root $\lambda=\omega\in\{-1,1\}$, which is related to $\zeta$ by
$\zeta=-\lambda\overline\alpha/|\alpha|$. No quasi-definite solution with constant $x_n$ can appear now, thus $s_0\neq\lambda$ and $x_0\neq-\zeta$
for any quasi-definite solution. The confluent form of the Tchebyshev polynomials $U_n(\omega)=(n+1)\lambda^n$ yields the quasi-definiteness
condition
\begin{equation} \label{QDb}
s_0 \neq \lambda \frac{n}{n+1}, \qquad n\geq0,
\end{equation}
i.e.,
$$
x_0 \neq -\zeta \frac{n}{n+1}, \qquad n\geq0,
$$
where, once $b_1$ is chosen, $x_0$ is fixed by (\ref{r0x0y0}) with $\beta=\lambda|\alpha|$.
\newline
\hspace*{10pt} If we fix only $\alpha$ and $\beta$, then $s_0$ can take any real value because now $B$ is non-negative on $\R$.
\newline
\hspace*{10pt} We have the relations
$$
\begin{array}{l}
\ds s_0U_n(\omega)-U_{n-1}(\omega) \sim (s_0-\lambda)n\lambda^n, \qquad b_2 = \frac{(b_1+\zeta)^2}{1-|b_1|^2},
\medskip \cr
\ds b_{n+1} = y_n \sim b_2\frac{x_0}{x_0+\zeta}\frac{\zeta^{n-1}}{n} =
-\overline\alpha\frac{b_1+\zeta}{\overline\alpha\overline{b}_1+\alpha\zeta}\frac{\zeta^{n-1}}{n},
\medskip \cr
\ds \lim b_n = \lim y_n = 0, \qquad \lim s_n = \lambda, \qquad \lim x_n = -\zeta,
\end{array}
$$
and the asymptotics of the corresponding orthogonal polynomials is
$$
\ba{l}
\ds \psi_{n+1}(z) \sim -\overline\alpha\frac{b_1+\zeta}{\overline\alpha\overline{b}_1+\alpha\zeta}\frac{\zeta^{n-1}}{n},
\qquad |z|<1,
\medskip \cr
\ds \psi_{n+1}(z) \sim \left(z-\zeta\right)z^n, \kern60pt |z|>1.
\ea
$$
\hspace*{10pt} We see that in this case there is a so well defined asymptotics for any solution as in (a). However, contrary to $|\omega|>1$, the
asymptotics of the frontier case $|\omega|=1$ defines no quasi-definite solution of the inverse problem.

\item[(c)] $A(z) = \alpha(z-\zeta_1)(z-\zeta_2), \quad \zeta_1 \neq \zeta_2, \quad \zeta_k\in\T \quad \Leftrightarrow \quad |\omega|<1$.

Now $B$ has two different roots $\lambda,\overline\lambda\in\T$ so that $\zeta_1=-\lambda\overline\alpha/|\alpha|$ and
$\zeta_2=-\overline\lambda\overline\alpha/|\alpha|$. The quasi-definiteness condition (\ref{QD-INV-LEB}) reads as
$$
s_0 \im\lambda^{n+1} \neq \im\lambda^n, \qquad n\geq0,
$$
that is,
$$
\overline{x}_0 (\zeta_1^{n+1}-\zeta_2^{n+1}) \neq \zeta_2^n-\zeta_1^n, \qquad n\geq0,
$$
which again can be considered as a constraint on $b_1$ due to (\ref{r0x0y0}).
\newline
\hspace*{10pt} Concerning the possible choices of $s_0$ when fixing only $\alpha$ and $\beta$, any real value of $s_0$ is possible since $B$ is now
positive on $\R$.
\newline
\hspace*{10pt} Analogously to case (b), $s_0\neq\lambda,\overline\lambda$ and $x_0\neq-\zeta_1,-\zeta_2$ for any quasi-definite solution. Writing
$\lambda=e^{i\theta}$, $\theta\notin\Z\pi$, and $s_0-\lambda=|s_0-\lambda|e^{i\gamma}$,
$$
s_0U_n(\omega)-U_{n-1}(\omega) = |s_0-\lambda|\frac{\sin((n+1)\theta+\gamma)}{\sin\theta},
$$
thus the quasi-definiteness condition can be stated as
$$
n\theta + \gamma \notin \Z\pi, \qquad n\geq1,
$$
and we find the identities
$$
\begin{array}{l}
\ds s_n = \frac{\sin((n+1)\theta+\gamma)}{\sin(n\theta+\gamma)} = \cos\theta + \frac{\sin\theta}{\tan(n\theta+\gamma)},
\medskip \cr
\ds b_{n+1} = y_n = \left(-\frac{\overline\alpha}{|\alpha|}\right)^n \frac{y_0}{|s_0-\lambda|}\frac{\sin\theta}{\sin(n\theta+\gamma)},
\end{array}
$$
which show that in this case $s_n$ and $|b_n|$ do not converge for any quasi-definite solution.

\end{itemize}

The algorithm (\ref{INV-LEB3}) giving the solutions of the inverse problem for the Lebesgue measure can be interpreted as a Newton algorithm to find
the zeros of a function. It is instructive to discuss the different behaviour of the associated Newton algorithm depending on the values of $\omega$
and $s_0$. This approach sheds light on the different asymptotics found in cases (a), (b) and (c). Since we will discuss the behaviour depending on
the values of $\omega$ and $s_0$, we remember that, given $P$, there is a set of permitted values $s_0$ and each choice of $s_0$ determines a circle
of solutions $C(P,s_0)$ which degenerates into a single solution when $s_0$ is a root of $B$. Remember also that the solutions of such a circle have
the same number of MOP.

The Newton algorithm for a real function $f(s)$ of a real variable $s$ is given by the iteration
$$
s_n=s_{n-1}-\frac{f(s_{n-1})}{f'(s_{n-1})}.
$$
Comparing this with (\ref{INV-LEB3}) we see that the algorithm providing the parameters $s_n$ of the inverse problem for the Lebesgue measure can be
understood as the Newton algorithm for a function $f(s)$ satisfying
$$
s-\frac{f(s)}{f'(s)} = 2\omega-\frac{1}{s}.
$$
Solving the above equation we find three cases ($\lambda_1$, $\lambda_2$ are the roots of $B$):

\begin{itemize}

\item[(a)] $\ds |\omega|>1 \Rightarrow f(s) = \left(\frac{|s-\lambda_2|^{\lambda_2}}{|s-\lambda_1|^{\lambda_1}}\right)^{\frac{1}{\lambda_2-\lambda_1}}$.

\item[(b)] $\ds |\omega|=1 \Rightarrow f(s) = |s-\omega|\exp\left(\frac{\omega}{\omega-s}\right)$.

\item[(c)] $\ds |\omega|<1 \Rightarrow f(s) = \sqrt{B(s)} \exp\left(\frac{\omega}{\sqrt{1-w^2}} \arctan\left(\frac{s-w}{\sqrt{1-w^2}}\right)\right)$.

\end{itemize}

\begin{figure}
\includegraphics{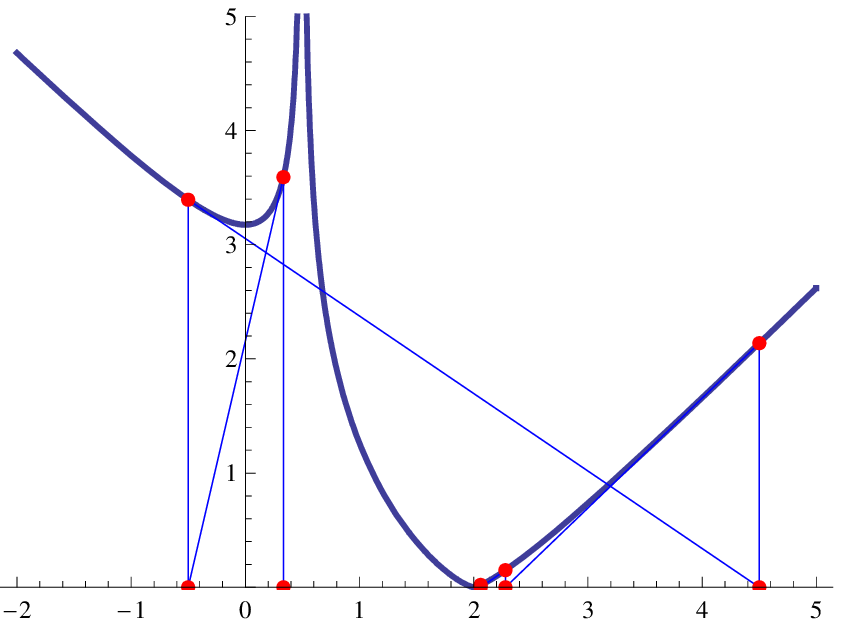}
\caption{\small (Case (a) - quasi-definite circle of solutions) First values of $s_n$ for $\omega=\frac{5}{4}$, $\lambda_1=\frac{1}{2}$,
$\lambda_2=2$, $\sigma_n=2\frac{4^n-1}{4^{n+1}-1}$, $s_0=\frac{1}{3}\notin\{\sigma_n\}$. This value of $s_0$ generates an infinite sequence $(s_n)$
such that $s_n\to\lambda_2^+$ monotonically for $n\geq2$. Hence, the solutions of the associated circle $C(P,s_0)$ are quasi-definite.} \vskip1cm
\includegraphics{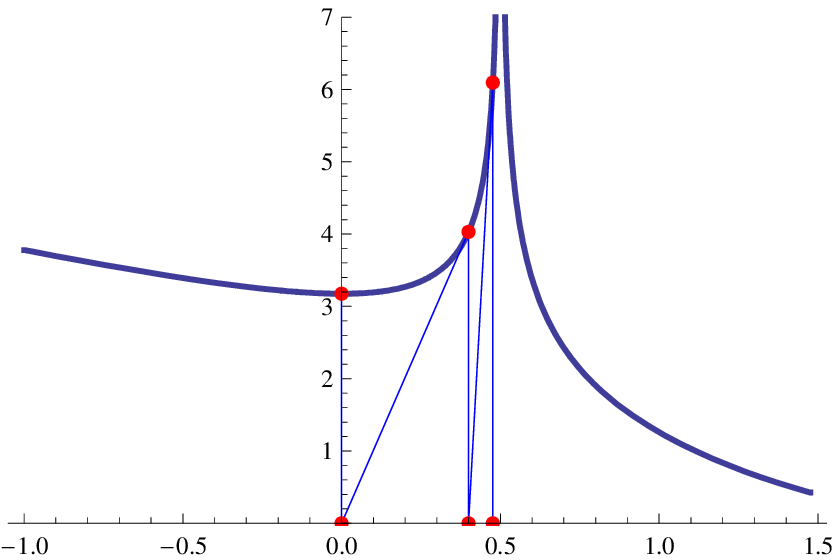}
\caption{\small (Case (a) - non quasi-definite circle of solutions) Values of $s_n$ for $\omega=\frac{5}{4}$, $\lambda_1=\frac{1}{2}$, $\lambda_2=2$,
$s_0=\sigma_2=\frac{10}{21}$. The iterations stop at $n=2$, thus the solutions of the circle $C(P,s_0)$ have only the MOP $\psi_0$, $\psi_1$,
$\psi_2$. Since the set $\{\sigma_n\}$ is infinite, there exist non quasi-definite solutions with an arbitrary number of MOP.}
\end{figure}
\begin{figure}
\includegraphics{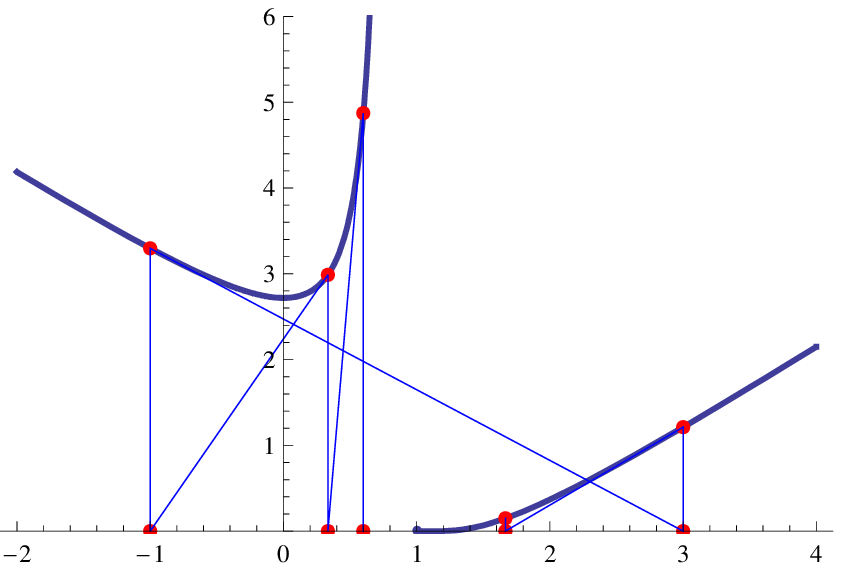}
\caption{\small (Case (b) - quasi-definite circle of solutions) First values of $s_n$ for $\omega=\lambda_1=\lambda_2=1$, $\sigma_n=\frac{n}{n+1}$,
$s_0=\frac{3}{5}\notin\{\sigma_n\}$. The situation is similar to Figure 1, but now $\lambda_1=\lambda_2$.} \vskip1cm
\includegraphics{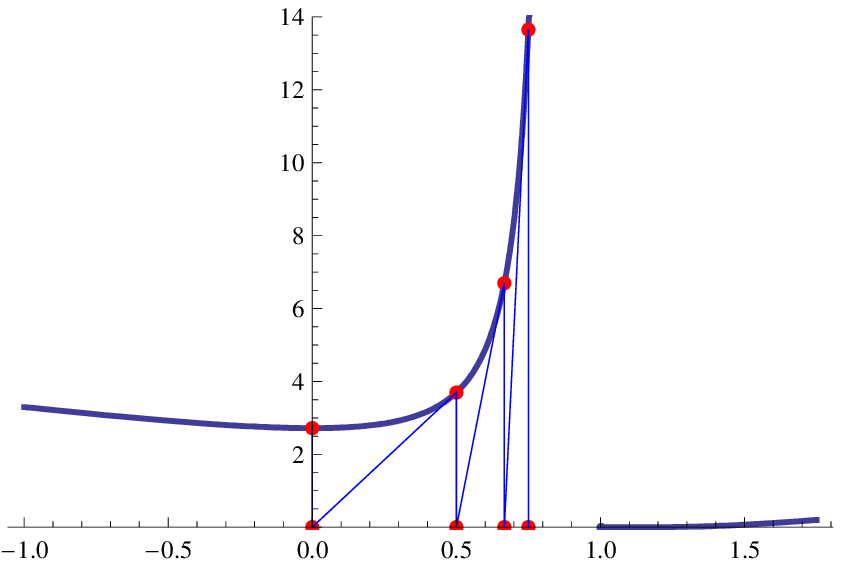} \caption{\small (Case (b) - non quasi-definite circle of solutions) Values of $s_n$ for $\omega=\lambda_1=\lambda_2=1$,
$s_0=\sigma_3=\frac{3}{4}$. The situation is similar to Figure 2 but now $\lambda_1=\lambda_2$ and we have chosen $s_0$ so that the solutions of the
circle $C(P,s_0)$ have four MOP.}
\end{figure}
\begin{figure}
\includegraphics{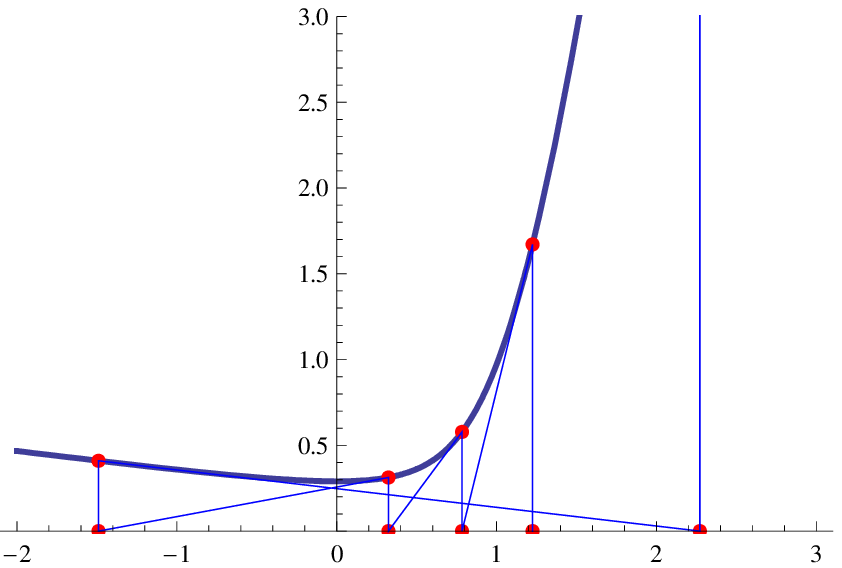}
\caption{\small (Case (c) - quasi-definite circle of solutions) First values of $s_n$ for $\omega=\frac{4}{5}$,
$\lambda_{1,2}=\frac{4}{5}\pm\frac{3}{5}i$, $\sigma_n=5\frac{\im(4+3i)^n}{\im(4+3i)^{n+1}}$, $s_0=\sqrt{\frac{3}{2}}\notin\{\sigma_n\}\subset\Q$. The
solutions of the associated circle $C(P,s_0)$ are quasi-definite because $s_0$ generates an infinite sequence $(s_n)$ which oscillates indefinitely
around the origin.} \vskip1cm
\includegraphics{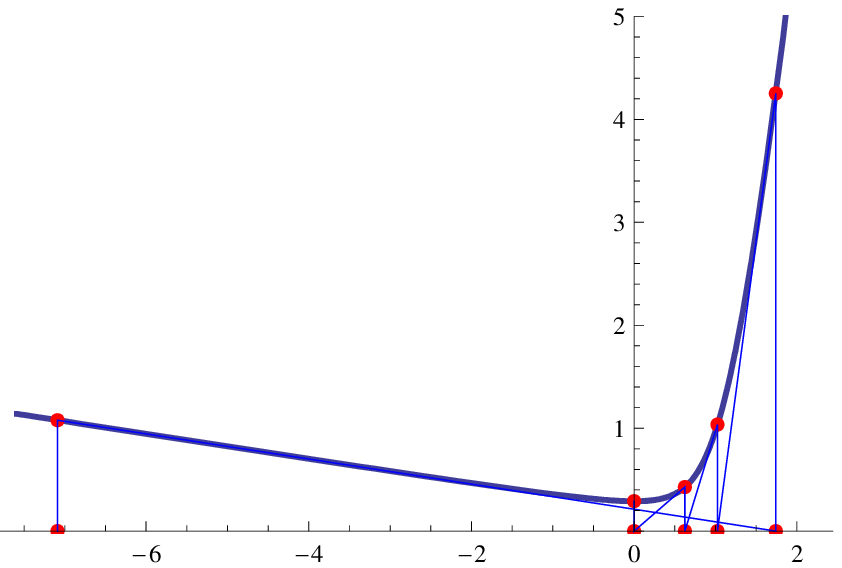}
\caption{\small (Case (c) - non quasi-definite circle of solutions) Values of $s_n$ for $\omega=\frac{4}{5}$,
$\lambda_{1,2}=\frac{4}{5}\pm\frac{3}{5}i$, $s_0=\sigma_4=-\frac{560}{79}$. The iterations stop at $n=4$, thus the solutions of the related circle
$C(P,s_0)$ have only five MOP. Like in Figures 2 and 4, the set $\{\sigma_n\}$ is infinite (but, on the contrary, $(\sigma_n)$ is not monotone
neither convergent) because $\lambda_{1,2}^2$ are not roots of the unity, so there exist non quasi-definite solutions with an arbitrary number of
MOP.}
\end{figure}
\begin{figure}
\includegraphics{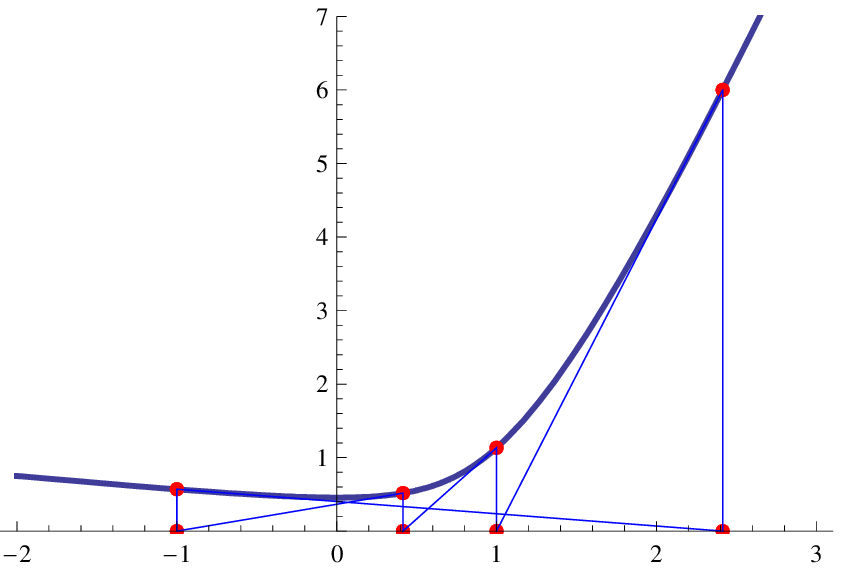}
\caption{\small (Case (c) - quasi-definite circle of solutions - periodic case) Values of $s_n$ for $\omega=\frac{1}{\sqrt{2}}$,
$\lambda_{1,2}=e^{\pm i\frac{\pi}{4}}$, $\sigma_n=\frac{\im(e^{i\frac{\pi}{4}n})}{\im(e^{i\frac{\pi}{4}(n+1)})}$,
$s_0=1\notin\{\sigma_n\}=\{0,\sqrt{2},1/\sqrt{2},\infty\}$. Like in Figure 5, the solutions of the associated circle $C(P,s_0)$ are quasi-definite
but, on the contrary, the sequences $(s_n)$ and $(\sigma_n)$ are periodic with period $4$ because $U_3(\omega)=0$.} \vskip1cm
\includegraphics{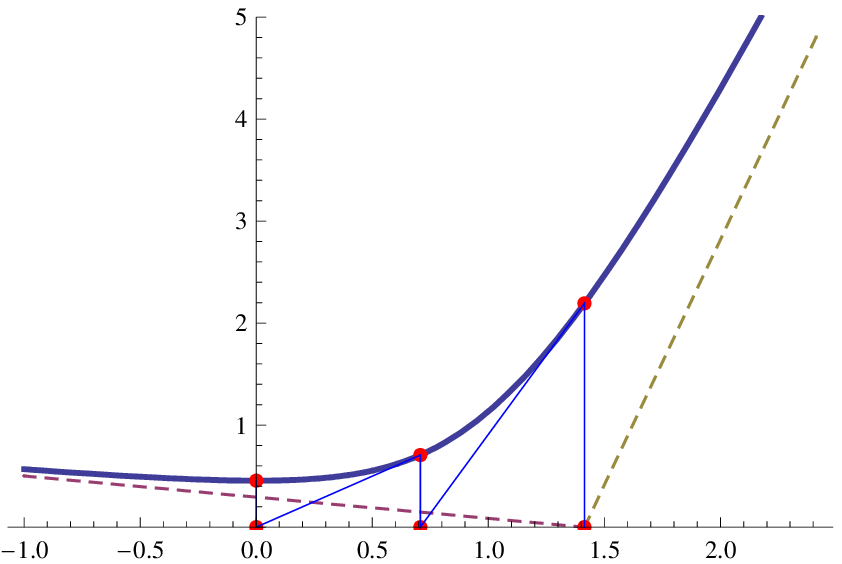}
\caption{\small (Case (c) - non quasi-definite circle of solutions - periodic case) Values of $s_n$ for $\omega=\frac{1}{\sqrt{2}}$,
$\lambda_{1,2}=e^{\pm i\frac{\pi}{4}}$, $s_0=\sigma_2=\frac{1}{\sqrt{2}}$. Like in Figure 6, the solutions of the circle $C(P,s_0)$ are non
quasi-definite, although in this case there exist only three MOP. Indeed, contrary to Figure 6, there is no non quasi-definite solution with more
than three MOP because $\sigma_n$ takes only three finite values: $\sigma_0=0$, $\sigma_1=\sqrt{2}$ and $\sigma_2=1/\sqrt{2}$. The picture, which can
be understood also as the inverse Newton algorithm starting at the origin which yields $(\sigma_n)$, shows clearly that $\sigma_3=\infty$ because the
corresponding tangent line becomes any of the two asymptotes.}
\end{figure}

The typical behaviour of the iterations in these three cases is shown in Figures 1 to 8, which represent the function $f(s)$ as well as some of these
iterations for different choices of $\omega$ and $s_0$. In any case the function $f(s)$ is analytic in $\R\setminus\{\lambda_1,\lambda_2\}$ and has a
minimum at $s=0$ which can stop the iterations, giving rise to a circle $C(P,s_0)$ of non quasi-definite solutions but with a finite segment of MOP
with the same length for all the circle.

When $|\omega|>1$ the function $f(s)$ diverges to $\infty$ at $s=\lambda$ and vanishes at $s=\lambda^{-1}$, where $\lambda$ is the root with smallest
module among $\lambda_1$, $\lambda_2$. Indeed $\lambda^{-1}$ is also the absolute minimum and, despite the visual effect in Figure 1 at
$\lambda^{-1}$, $f \in C^{(1}(\{\lambda^{-1}\})$ so $f'(\lambda^{-1})=0$. Excluding the case $s_0=\lambda$, the iterations, which must start at a
point of $(-\infty,\lambda_1]\cup[\lambda_2,\infty)$, always converge to $\lambda^{-1}$ (corresponding to a circle of quasi-definite solutions) or
they stop at the origin after a finite number of steps (corresponding to a circle of solutions with only a finite segment of MOP).

If $\omega=\pm1$, then $\lim_{s\to\lambda^\mp}f(s)=\infty$ and $\lim_{s\to\lambda^\pm}f(s)=\lim_{s\to\lambda^\pm}f'(s)=0$, where
$\lambda=\lambda_1=\lambda_2=\pm1$, which plays again the role of an attractor where the iterations converge (circle of quasi-definite solutions)
while they do not stop at the origin (circle of solutions with a finite segment of MOP).

On the contrary, $f(s)$ has no divergence neither zero when $|\omega|<1$, and the origin is then the absolute minimum. In this case, as far as the
iterations do not reach the origin (circle of solutions with a finite segment of MOP), they oscillate indefinitely around such a minimum (circle of
quasi-definite solutions).

In any case, for each value of $\omega$, the values of $s_0$ associated to non quasi-definite solutions can be obtained by the inverse Newton
algorithm starting at the origin, so they form a sequence $(\sigma_n)$ given by
\begin{equation} \label{sn}
\sigma_n=\frac{1}{2\omega-\sigma_{n-1}}, \qquad \sigma_0=0.
\end{equation}
If $s_0=\sigma_n$, then $s_j=\sigma_{n-j}\neq0$ for $j<n$ and $s_n=0$, hence the solutions of the related circle $C(P,s_0)$ have only $n+1$ MOP. When
$|\omega|\geq1$, $(\sigma_n)$ is a monotone sequence with limit $\lambda$, but if $|\omega|<1$ then $(\sigma_n)$ is non convergent and oscillates
around the origin. Eventually, $\sigma_{n-1}=2\omega$ and the above iterations stop. To understand this fact notice that (\ref{QD-INV-LEB}) shows
that
$$
\sigma_n=\frac{U_{n-1}(\omega)}{U_n(\omega)}
$$
if $U_n(\omega)\neq0$, otherwise $\sigma_n$ has no meaning because no value of $s_0$ can satisfy $s_0U_n(\omega)=U_{n-1}(\omega)$ when
$U_n(\omega)=0$. The recurrence for $U_n$ implies that $\sigma_{n-1}=2\omega$ iff $U_n(\omega)=0$, so this is exactly the case where $\sigma_n$ does
not exist and, besides, $\sigma_{n+1}=0=\sigma_0$, hence the values of $\sigma_j$, $j \geq n+1$, are simply a reiteration of the values for
$j=-1,0,\dots,n-1$ if we define $\sigma_{-1}=\infty$. Therefore, (\ref{sn}) always works for $n\geq-1$ if we assume that $\sigma_{n-1}=2\omega$ gives
$\sigma_n=\infty$, which leads to $\sigma_{n+1}=0$ and yields a periodic sequence $(\sigma_j)$ in $\R\cup\{\infty\}$ with period $n+1$.

Summarizing, if $\omega$ is a zero of $U_n$, which can hold only when $|\omega|<1$, there is a finite number of non quasi-definite circles of
solutions $C(P,s_0)$, those ones related to the initial values $s_0\in\{\sigma_j\}_{j=1}^{n-1}$. Furthermore, if $n$ is the smallest index such that
$U_n(\omega)=0$, the quantities $\sigma_j$, $j=0,\dots,n-1,$ are different from each other, hence there are exactly $n-1$ non quasi-definite circles
$C(P,s_0)$, and the length of the corresponding finite segments of MOP runs from 2 to $n$ when $s_0=\sigma_1,\dots,\sigma_{n-1}$. Therefore, there
are no non quasi-definite solutions with more than $n$ MOP.

On the contrary, if $U_n(\omega)\neq0$ for all $n$, then $\sigma_j \neq \sigma_k$ for $j \neq k$, thus an infinite denumerable set of non
quasi-definite circles $C(P,s_0)$ appear, which correspond to $s_0\in\{\sigma_j\}_{j=0}^\infty$. In this case, given any $n\in\N$, there is exactly
one non quasi-definite circle of solutions with only $n+1$ MOP, which corresponds to $s_0=\sigma_n$.

As a final remark notice that $U_n(\omega)=0$ means $\lambda^{2n+2}=1$, $\lambda\neq\pm1$. Therefore, not only the sequence $(\sigma_j)$, but also
$(U_j(\omega))$ is in this case periodic with period $n+1$, so $(s_j)$ shows such a periodic behaviour too no matter the choice of $s_0$.

\section{Applications of these techniques} \label{A}

The characterization we have obtained for hermitian functionals related by polynomial perturbations is not only interesting by itself, but provides
an efficient tool to answer different questions concerning orthogonal polynomials on the unit circle. In this section we will show two examples of
this. The first one exploits the fact that a polynomial perturbation is equivalent to a linear relation with polynomial coefficients between two
sequences of orthogonal polynomials and their reversed ones. The second one deals with a problem concerning associated polynomials, which can be
solved due to the formulation of a polynomial perturbation in terms of a difference equation for two sequences of Schur parameters.

\subsection{Orthogonal polynomials and linear combinations with constant polynomial coefficients}

There are in the literature different results on the orthogonality properties of linear combinations of orthogonal polynomials. In particular, it is
known that, if $(\varphi_n)$ and $(\psi_n)$ are MOP on the unit circle, a relation like
\begin{equation} \label{LC}
\kern-5pt \psi_{n+r} = \sum_{j=0}^r (\lambda_{j,n}\varphi_{n+j} + \kappa_{j,n}\varphi_{n+j}^*),
\kern9pt \lambda_{j,n},\kappa_{j,n}\in\C, \kern9pt \lambda_{0,n}\neq0, \kern9pt n\geq0,
\end{equation}
forces $(\psi_n)$ to be Bernstein-Szeg\H{o} polynomials when $r>1$ (see \cite{MaPeSt96}). The result is so strong that it holds assuming (\ref{LC})
only when $n \geq n_0$ for some $n_0$, and even if we suppose that the sum in (\ref{LC}) is up to and index $r(n)$ depending on $n$, with the simple
restriction $1 < r(n) \leq n/2$ for $n \geq n_0$ (see \cite{MaPeSt97}).

A way to escape from this triviality is to consider a more general relation than (\ref{LC}). Identity (\ref{LC}) implies that
$\psi_{n+r}\in(z\PP_{n-2})^{\bot_{n+r}}\subset(z^r\PP_{n-r-1})^{\bot_{n+r}}$ for $r\geq1$, where the orthogonality is understood with respect to the
functional associated with $(\varphi_n)$. Thus, Lemma \ref{BASISOC} shows that (\ref{LC}) is a particular case of
\begin{equation} \label{PC}
\psi_{n+r} = X_n\varphi_n + Y_n\varphi_n^*, \qquad X_n\in\PP_r, \qquad Y_n\in\PP_{r-1}, \qquad n\geq0.
\end{equation}
However, contrary to (\ref{LC}), a relation like (\ref{PC}) can hold for non trivial MOP $(\varphi_n)$ and $(\psi_n)$, since it is always equivalent
to a polynomial perturbation relation between the corresponding orthogonality functionals due to Theorem \ref{EQGENERALPROBLEMTHMF} and the
subsequent comments, together with Proposition \ref{RED}: the hermitian functionals $u$ and $v$ associated with $(\varphi_n)$ and $(\psi_n)$ must be
related by $u=vL$ where $L=P+P_*$ is given by a polynomial $P$ with $\deg P \leq r$; the condition $X_n(0)\neq0$, which holds for no $n$ or
simultaneously for all $n$, characterizes the case $\deg P = r$.

In this section we will show that the freedom enclosed in (\ref{PC}) is large enough to yield non trivial solutions even when imposing very strong
conditions on $X_n$ and $Y_n$. More precisely, we will find all the pairs of sequences of MOP $(\psi_n)$ and $(\varphi_n)$ related by (\ref{PC}) with
constant polynomials coefficients, i.e.,
\begin{equation} \label{CTE}
\psi_{n+r}=X\varphi_n+Y\varphi_n^*, \quad X\in\PP_r, \quad Y\in\PP_{r-1}, \quad n\geq0.
\end{equation}

This is not only an academic problem, but its importance relies on the fact that the constant solutions should play the role of fixed points with
respect to the asymptotics of the polynomials $X_n$, $Y_n$ related to the quasi-definite solutions of $H_r(u)$. Therefore, some of these fixed points
should act as attractors whose study could give information about the asymptotics for the quasi-definite solutions of $H_r(u)$, similarly to what
happens in Example \ref{EX-INV}.

Relation (\ref{CTE}) can be rewritten, together with its reversed, as
$$
\Psi_{n+r} = \mathcal{X} \Phi_n, \qquad \cX=\pmatrix{X & Y \cr zY^* & X^*}, \qquad n\geq 0,
$$
and the polynomial perturbation is recovered by $A=\det\cX$.

As follows from Theorem \ref{EQUIVALENCESRRGC} and Proposition \ref{RED}, the problem we want to solve is equivalent to the recurrence
$\mathcal{T}_{n+r} \cX = \cX \mathcal{S}_n$, $n\geq1$, and the initial condition $\cX\Phi_0=\Psi_r$, i.e.,
\begin{equation} \label{SISTCTE}
\begin{array}{l}
\cases{
\overline{a}_n Y = b_{n+r} Y^*,
\smallskip \cr
a_n X - b_{n+r} X^*= (z-1) Y,
}
\qquad n\geq1,
\medskip \cr \kern10pt
\psi_r=X+Y, \quad X\in\PP_r, \quad Y\in\PP_{r-1}.
\end{array}
\end{equation}

If $Y=0$, equations (\ref{SISTCTE}) yield $b_{n+r}X^*=a_nX$ and $\psi_r=X$. Since $\psi_r$ and $\psi_r^*$ have no common roots, we find that
$a_n=b_{n+r}=0$ for $n\geq1$. This situation corresponds to $u$ being the functional associated with the Lebesgue measure and MOP $\varphi_n(z)=z^n$,
and $v$ a Bernstein-Szeg\H{o} type functional with the first $r+1$ MOP generated by arbitrary Schur parameters $b_1,\dots,b_r\in\C\setminus\T$, while
$\psi_{n+r}(z)=z^n\psi_r(z)$ for $n\geq1$.

Let us find now the solutions with $Y\neq 0$. Denote for convenience $a=a_n$ and $b=b_{n+r}$. The first equation of (\ref{SISTCTE}) simply says that
$Y$ is proportional to a self-reciprocal polynomial in $\PP_{r-1}$ and $|b|=|a|$. Using such equation and bearing in mind that $\psi_r = X+Y$ and
$\psi_r^* = X^*+zY^*$, we can eliminate $X$ and $X^*$ in the second equation of (\ref{SISTCTE}), which becomes
\begin{equation} \label{SISTECT3}
a \psi_r - b \psi_r^* = \left[z(1-\overline{a})-(1-a)\right]Y.
\end{equation}
Therefore,
\begin{equation} \label{SISTECT2}
\ba{l}
\ds b = a \frac{\psi_r(\zeta)}{\psi_r^*(\zeta)}, \qquad \zeta=\frac{1-a}{1-\overline{a}},
\medskip \cr
\ds Y(z) = \frac{a}{1-\overline{a}} \frac{1}{\psi_r^*(\zeta)} \frac{\psi_r^*(\zeta)\psi_r(z)-\psi_r(\zeta)\psi_r^*(z)}{z-\zeta} =
\medskip \cr
\kern26pt \ds = \frac{a}{1-\overline{a}} \, \varepsilon_r \overline{\left(\frac{\zeta}{\psi_r(\zeta)}\right)}K_{r-1}(z,\zeta), \ea
\end{equation}
where we have used the Christoffel-Darboux formula for the $n$-th kernel
$K_n(z,\zeta)=\sum_{j=0}^n\varepsilon_j^{-1}\psi_j(z)\overline{\psi_j(\zeta)}$ associated with the MOP $(\psi_j)$.

As a consequence, given $\psi_r$, the solutions of (\ref{SISTCTE}) are determined by an arbitrary choice of $a\in\C\setminus\T$: (\ref{SISTECT2})
provides $b$ and $Y$ self-reciprocal in $\PP_{r-1}$ up to a factor, solving the first equation of (\ref{SISTCTE}), and finally $X=\psi_r-Y$ solves
the second equation of (\ref{SISTCTE}).

On the other hand, given $X$, $Y$, let us see how many solutions $a$, $b$ of (\ref{SISTCTE}) we can expect. If we suppose two different solutions
$a$, $b$ and $a'$, $b'$, (\ref{SISTCTE}) gives
\begin{equation} \label{UNIQUE}
\cases{
(\overline{a}-\overline{a}') Y = (b-b') Y^*,
\smallskip \cr
(a-a') X = (b-b') X^*.
}
\end{equation}
Then, $Y^* \propto Y$, $X^* \propto X$ and, using again (\ref{SISTCTE}), we find that $Y=0$ or $X \propto (z-1)Y$. In the first case $\psi_r=X$,
which is not possible because $X^* \propto X$. In the second case $Y$ divides $\psi_r=X+Y$, which implies that $Y$ is a constant because $Y^* \propto
Y$. Hence, $X(z)=z-1$ and the polynomial modification must be of degree $r=1$.

As a conclusion, given $X$, $Y$, the equations (\ref{SISTCTE}) have at most one solution $a$, $b$ when the degree $r$ of the modification is greater
than 1, or when it is equal to 1 but $X(z) \neq z-1$. Thus, concerning the MOP related by (\ref{CTE}) we have to distinguish two cases depending on
the degree $r$ of the modification.

\begin{itemize}

\item $r>1$

In this case, given $X$, $Y$, the Schur parameters $a_n$, $b_{n+r}$ must be constants of equal modulus for $n\geq1$: the unique solution $a$, $b$ of
equation (\ref{SISTCTE}). Furthermore, for any choice of $a,b_1,\dots,b_r\in\C\setminus\T$ the system (\ref{SISTCTE}) has a unique solution in $X$,
$Y$, $b$ obtained through (\ref{SISTECT2}) and the relation $X=\psi_r-Y$. In other words, the MOP related by (\ref{CTE}) are those $(\varphi_n)$
corresponding to a sequence of constant Schur parameters $(a,a,\dots)$ and those $(\psi_n)$ related to a sequence $(b_1,\dots,b_r,b,b,\dots)$ of
Schur parameters, where $a,b_1,\dots,b_r\in\C\setminus\T$ are arbitrary and $b$ is given by (\ref{SISTECT2}). The MOP related by (\ref{CTE}) are thus
parametrized by $a,b_1,\dots,b_r\in\C\setminus\T$.

\item $r=1$

If $X(z) \neq z-1$ the conclusions are similar to those corresponding to $r>1$. However, when $X(z)=z-1$ the system (\ref{SISTCTE}) has infinitely
many solutions no matter the choice of $Y=y\in\C$. To see this, let us write (\ref{SISTCTE}) explicitly,
$$
\cases{
\overline{a} y = b \overline{y},
\cr
a+b = y,
\cr
b_1=y-1.
}
$$
Since $b_1\notin\T$ forces $y\neq0$, the solutions $a$, $b$ are all the symmetric points of the perpendicular bisector $\Pi(y)$ of the segment
$[0,y]$. Therefore, the solutions corresponding to $X(z)=z-1$ can be construct in the following way: choose $b_1\in\C\setminus\T$, which determines
$y=b_1+1$; for each $n\geq1$ choose $a_n \in \Pi(y)\setminus\T$ and $b_{n+1} \in \Pi(y)$ as its symmetric point with respect to the segment $[0,y]$.
This procedure generates all the sequences of Schur parameters $(a_n)$, $(b_n)$ whose MOP $(\varphi_n)$, $(\psi_n)$ are related by
$$
\psi_{n+1}(z) = (z-1) \varphi_n(z) + y \varphi_n^*(z), \qquad y\in\C.
$$
Hence, the solutions with $X(z)=z-1$ are parametrized by $b_1\in\C\setminus\T$ and an infinite sequence $(a_1,a_2,\dots)$ lying on
$\Pi(1+b_1)\setminus\T$.
\newline
\hspace*{10pt} On the other hand, the solutions with $X(z) \neq z-1$ are parametrized by $b_1,a\in\C\setminus\T$ with $a \notin \Pi(1+b_1)$, and the
corresponding pair of sequences of Schur parameters is given by $(a,a,\dots)$ and $(b_1,b,b,\dots)$ with $b=a(\zeta+b_1)/(1+\overline{b}_1\zeta)$.
This yields all the MOP related by
$$
\psi_{n+1}(z) = (z+x) \varphi_n(z) + y \varphi_n^*(z), \qquad x,y\in\C, \qquad x\neq-1.
$$
Moreover, from this equality for $n=0$ and (\ref{SISTECT2}) we find that the parameters $x$, $y$ related to a choice of $b_1$ and $a$ are
\begin{equation} \label{xyb1a}
x = b_1-y, \qquad y = \frac{a(1-|b_1|^2)}{(1-\overline{a})+\overline{b}_1(1-a)}.
\end{equation}

\end{itemize}

Concerning the possible values of the polynomials $X$ and $Y$, we have to point out that $Y$ must be proportional to a self-reciprocal polynomial in
$\PP_{r-1}$, as follows from (\ref{SISTCTE}). Indeed, (\ref{SISTECT2}) shows that $Y(z)$ is proportional to a kernel $K_{r-1}(z,\zeta)$ for some
$\zeta\in\T$, thus it has exact degree $r-1$ unless $Y=0$. On the other hand, $X$ is a monic polynomial of degree $r$ which can not be proportional
to a self-reciprocal one unless $r=1$ and $X(z)=z-1$, as follows from the reasoning in the paragraph after (\ref{UNIQUE}). This, together with the
fact that $\psi_r=X+Y$ must be an orthogonal polynomial, are necessary conditions which must be fulfilled by the polynomial coefficients $X$, $Y$.
Nevertheless, they are not sufficient conditions for the existence of MOP satisfying (\ref{CTE}). To see this consider the case $r=1$, where these
conditions become
\begin{equation} \label{NC}
X(z)=z+x, \quad Y(z)=y, \quad x\in\C\setminus\T\cup\{-1\}, \quad x+y\in\C\setminus\T.
\end{equation}
However, solving (\ref{xyb1a}) for $b_1$ and $a$ we get
$$
b_1=x+y, \qquad a=y\frac{1+\overline{x}}{1-|x|^2},
$$
which shows that to get the alluded necessary and sufficient conditions for $r=1$ we must add to (\ref{NC}) the following one
$$
|y|\neq\left|\frac{1-|x|^2}{1+x}\right| \quad \mathrm{if} \quad |x| \neq 1.
$$

Concerning the polynomial perturbation $L=P+P_*$ such that $u=vL$, we know that $A \propto XX^*-zYY^*$. Hence, when $X(z)=z-1$ we find that $P(z)
\propto z + (|y|^2/2-1)$. As for the rest of solutions, related to Schur parameters $(a,a,\dots)$, $(b_1,\dots,b_r,b,b,\dots)$ with $b$ given in
(\ref{SISTECT2}), we only know that $\deg P \leq r$. The inequality $\deg P < r$ is characterized by any of the statements of the following
equivalence, which follows from the previous results and the recurrence for $(\psi_n)$,
$$
\begin{array}{c}
\deg P < r \Leftrightarrow X(0)=0 \Leftrightarrow Y(0)=b_r \Leftrightarrow b=b_r \Leftrightarrow
\medskip \cr
\ds \Leftrightarrow b_r=a\frac{\psi_r(\zeta)}{\psi_r^*(\zeta)} \Leftrightarrow b_r=a\frac{\psi_{r-1}(\zeta)}{\psi_{r-1}^*(\zeta)}.
\end{array}
$$
That is, among the values of $a,b_1,\dots,b_r$ which parametrize the solutions with $X(z) \neq z-1$, the inequality $\deg P < r$ holds for those ones
with $b_r$ determined by $a,b_1,\dots,b_{r-1}$ through $b_r=a\psi_{r-1}(\zeta)/\psi_{r-1}^*(\zeta)$. The solutions with $\deg P < r$ correspond to
$b_n=b$ for $n \geq r$, while the solutions with $\deg P = r$ are those ones with $(b_n)$ given by $(b_1,\dots,b_r,b,b,\dots)$, $b_r \neq b$. Notice
that each solution with $\deg P < r$ has a sequence $(b_n)$ with the form $(b_1,\dots,b_s,b,b,\dots)$, $b_s \neq b$, for some $s<r$, and then $\deg P
= s$ and one can find new polynomial coefficients $\hat X\in\PP_s$, $\hat Y\in\PP_{s-1}$ such that $\psi_{n+s}=\hat{X}\varphi_n+\hat{Y}\varphi_n^*$,
$n\geq0$. In any case, $b=a\psi_j(\zeta)/\psi_j^*(\zeta)$ for $j \geq \deg P$.

\subsection{Associated polynomials and polynomial modifications}

Given a sequence $(\psi_n)$ of MOP with Schur parameters $(b_n)$, the associated polynomials are those MOP $(\varphi_n)$ with Schur parameters
$(a_n)$, $a_n=b_{n+1}$. Despite the similarity of their Schur parameters, the corresponding orthogonality functionals can be quite different. We will
consider the following question concerning such functionals: when is the functional $u$ of the associated polynomials $(\varphi_n)$ a polynomial
modification of the functional $v$ related to the original MOP $(\psi_n)$? We will answer explicitly this question for a polynomial modification of
degree 1.

According to Theorem \ref{EQUIVALENCESRRGC}, this is equivalent to the existence of matrices $\cC_n\in\J_1$ such that $\mathcal{C}_n
\mathcal{B}_{n+1} = \mathcal{A}_n \widetilde{\mathcal{C}}_{n-1}$, $\mathcal{B}_{n+1}=\mathcal{A}_n$, $n\geq1$, with $\cC_0\in\J_1^\reg$ satisfying
the initial condition $\cC_0\Psi_1=A\Phi_0$. Let us denote $P(z) =\alpha z + \beta$, $\alpha\in\C^*$, $\beta\in\R$. The recurrence for $\cC_n$ can be
written as
$$
\pmatrix{
\alpha z + c_n & d_n
\cr
z\overline{d}_n & c_nz+\overline{\alpha}
}
\pmatrix{
1 & a_n
\cr
\overline{a}_n & 1
}
=
\pmatrix{
1 & a_n
\cr
\overline{a}_n & 1
}
\pmatrix{
\alpha z + c_{n-1} & zd_{n-1}
\cr
\overline{d}_{n-1} & c_{n-1}z+\overline{\alpha}
},
$$
for some coefficients $c_n\in\R^*$, $d_n\in\C$. Splitting this matrix recurrence gives the equivalent system of equations
\begin{equation} \label{EQ-ASS}
\cases{
c_n + \overline{a}_n d_n = c_{n-1} + a_n \overline{d}_{n-1},
\cr
a_n c_n + d_n = \overline{\alpha} a_n ,
\cr
\alpha a_n = a_n c_{n-1} + d_{n-1}.
}
\end{equation}

Taking determinants in the matrix recurrence and setting $z=0$, we find that $c_n=c_{n-1}$ for $n\geq 1$, so $c_n=c_0$ for $n \geq 0$. Therefore,
(\ref{EQ-ASS}) reads as
\begin{equation} \label{EQ-ASS2}
\cases{
\overline{a}_n d_n = a_n \overline{d}_{n-1},
\cr
a_n \left(\overline{\alpha} - c_0\right) = d_n,
\cr
a_n \left(\alpha - c_0\right) = d_{n-1},
}
\end{equation}
although the first equation is a consequence of the others.

Assume that $\alpha=c_0$. Then, $d_n=0$ for all $n$ and the initial condition is $A=\alpha(z+1)(z+b_1)$, which is not possible because $A$ is
self-reciprocal while $|b_1|\neq1$. Hence, $\alpha \neq c_0$ and the solution of (\ref{EQ-ASS2}) is
$$
a_{n+1} = \lambda^n a_1,
\qquad
d_n = \lambda^n(\alpha-c_0)a_1,
\qquad
\lambda={\overline{\alpha}-c_0 \over \alpha-c_0},
\qquad
n\geq 0.
$$
Besides, the initial condition
$$
\alpha z^2 + 2\beta z + \bar{\alpha} = \left(\alpha z + c_0\right)(z+b_1) + d_0(\overline{b}_1z + 1)
$$
yields the parameters of the polynomial perturbation,
$$
\alpha = \overline{b}_1 c_0 + \overline{d}_0,
\qquad
\beta = \frac{1}{2}(\alpha b_1 + \overline{b}_1 d_0 + c_0) = \frac{c_0}{2}(1-|b_1|^2) + \re(\alpha b_1).
$$

Taking into account that $d_0=(\alpha-c_0)a_1$, we can express $\alpha$, $\beta$, $\lambda$, $d_0$, in terms of $a_1$, $b_1$, $c_0$,
$$
\begin{array}{l}
\ds \alpha = c_0 \frac{\overline{a}_1(b_1-a_1)+(\overline{b}_1-\overline{a}_1)}{1-|a_1|^2},
\medskip \cr
\ds \beta = c_0 \left\{\frac{1}{2}(1-|b_1|^2)+\frac{\re[(\overline{a}_1(b_1-a_1)+(\overline{b}_1-\overline{a}_1))b_1]}{1-|a_1|^2}\right\},
\medskip \cr
\ds \lambda = \frac{(b_1-1)+a_1(\overline{b}_1-1)}{(\overline{b}_1-1)+\overline{a}_1(b_1-1)},
\medskip \cr
\ds d_0 = c_0 a_1 \frac{(\overline{b}_1-1)+\overline{a}_1(b_1-1)}{1-|a_1|^2}.
\end{array}
$$
The fact that $\deg P = 1$ means that $\alpha\neq0$. This only excludes the possibility $a_1 = b_1$, which gives $\lambda=1$ and thus corresponds to
the trivial case $a_n=b_n$ for all $n$, i.e., $u=v$.

Therefore, the arbitrariness in $c_0\in\R^*$ is simply the freedom of the polynomial perturbation in a multiplicative real factor, and the solutions
of the problem are parametrized by $a_1,b_1\in\C\setminus\T$ with $a_1 \neq b_1$: the MOP $(\psi_n)$ whose associated ones $(\varphi_n)$ come from a
polynomial perturbation of degree 1 of the orthogonality functional of $(\psi_n)$ are those ones with Schur parameters
$(b_1,a_1,a_1\lambda,a_1\lambda^2,\dots)$, where $\lambda\in\T$ is the square of the phase of $(b_1-1)+a_1(\overline{b}_1-1)$. The associated
polynomials $(\varphi_n)$ have Schur parameters $(a_1,a_1\lambda,a_1\lambda^2,\dots)$, so they are obtained by a rotation
$\varphi_n(z)=\lambda^n\phi_n(\overline\lambda z)$ of the MOP $(\phi_n)$ with constant Schur parameters $(a_1,a_1,a_1,\dots)$.

We can use $\alpha$, $\beta$ and $b_1$ as free parameters too. The initial condition can be expressed as
$$
\pmatrix{1 & \overline{b}_1 \cr b_1 & 1} \pmatrix{c_0 \cr d_0} = \pmatrix {\beta-\alpha b_1 \cr \overline{\alpha}},
$$
with solutions
$$
c_0 = 2 {\beta - \re(\alpha b_1) \over 1-|b_1|^2},
\qquad
d_0 = {A(-b_1) \over 1-|b_1|^2}.
$$
This gives
$$
\begin{array}{l}
\ds a_1 = \frac{d_0}{\alpha-c_0} = \frac{A(-b_1)}{\alpha(1-|b_1|^2)-2(\beta-\re(\alpha b_1))},
\medskip \cr
\ds \lambda = \frac{\overline\alpha(1-|b_1|^2)-2(\beta-\re(\alpha b_1))}{\alpha(1-|b_1|^2)-2(\beta-\re(\alpha b_1))},
\end{array}
$$
providing a solution whenever $c_0 \neq 0,\alpha$ and $|a_1|\neq1$, i.e.,
$$
\beta \neq \re(\alpha b_1), \frac{\alpha}{2}(1-|b_1|^2)+\re(\alpha b_1),
\quad
\left|\frac{A(-b_1)}{\alpha(1-|b_1|^2)-2(\beta-\re(\alpha b_1))}\right|\neq1.
$$

\bigskip

\noindent{\bf Acknowledgements}

\medskip

This work was partially supported by the Spanish grants from the Ministry of Education and Science, project code MTM2005-08648-C02-01, and the
Ministry of Science and Innovation, project code MTM2008-06689-C02-01, and by Project E-64 of Diputaci\'{o}n General de Arag\'{o}n (Spain).

\end{document}